\documentclass[12pt,a4paper]{amsart}

\usepackage{amsmath, nicefrac, amsthm, verbatim, amsfonts, amssymb, xcolor, bbm,mathtools}
\usepackage{graphics, xspace, enumerate}
\usepackage[top=25mm,bottom=25mm,left=25mm,right=25mm]{geometry}	

\usepackage{dsfont}
\usepackage[bb=boondox]{mathalfa}

\usepackage{graphicx, amsmath}
\usepackage[colorlinks=true,
citecolor=green,urlcolor=green,linkcolor=blue,
bookmarksopen=true,unicode=true,pdffitwindow=true]{hyperref}
\usepackage[english]{babel}
\RequirePackage[mathscr]{eucal}
\usepackage{seqsplit,cleveref}
\usepackage{todonotes}
\usepackage{xcolor}

\hypersetup{pdfauthor={}}
\hypersetup{pdftitle={The growth  of the range of stable random walks}}

\theoremstyle{plain}
\newtheorem{theorem}{Theorem}[section]
\newtheorem{corollary}[theorem]{Corollary}
\newtheorem{lemma}[theorem]{Lemma}

\theoremstyle{definition}

\newcommand {\Prob} {\ensuremath{\mathbb{P}}}
\newcommand {\R} {\ensuremath{\mathbb{R}}}

\newcommand {\N} {\ensuremath{\mathbb{N}}}

\newcommand{\D}{\mathrm{d}}
\newcommand {\X} {\ensuremath{\mathrm{X}}}

\newcommand{\Capa}{\operatorname{Cap}}
\newcommand{\Var}{\operatorname{Var}}
\newcommand{\E}{\mathrm{e}}

\newcommand{\bbN}{\mathbb{N}}
\newcommand{\bbZ}{\mathbb{Z}}
\newcommand{\bbR}{\mathbb{R}}

\newcommand{\bbP}{\mathbb{P}}
\newcommand{\bbE}{\mathbb{E}}
\newcommand{\bbjedan}{\mathbbm{1}}
\newcommand{\sB}{{\mathscr{B}}}

\newcommand{\Ind}{\mathds{1}}

\newcommand{\calH}{\mathcal{H}}

\newcommand{\calI}{\mathcal{I}}

\newcommand{\calF}{\mathcal{F}}

\newcommand{\calD}{\mathcal{D}}

\newcommand{\calR}{\mathcal{R}}
\newcommand{\calN}{\mathcal{N}}
\newcommand{\calS}{\mathcal{S}}
\newcommand{\calV}{\mathcal{V}}
\newcommand{\calG}{\mathcal{G}}

\newcommand{\aps}[1]{\vert #1 \vert}
\newcommand{\APS}[1]{\left\vert #1 \right\vert}
\newcommand{\norm}[1]{\lVert #1 \rVert}
\newcommand{\floor}[1]{\lfloor #1 \rfloor}

\newcommand{\OBL}[1]{\left( #1 \right)}

\newcommand{\UGL}[1]{\left[ #1 \right]}

\newcommand{\VIT}[1]{\left\{ #1 \right\}}

\numberwithin{equation}{section}

\title[Limit theorems for a stable sausage]{Limit theorems for a stable sausage} 

\author[W.\ Cygan]{Wojciech Cygan}
\address[Wojciech Cygan]{Institut f\"{u}r Mathematische Stochastik\\Technische Universit\"{a}t Dresden\\Dresden\\Germany
\& 
Instytut Matematyczny\\Uniwersytet Wroc\l{}awski\\ Wroc\l{}aw\\ Poland}
\email{wojciech.cygan@uwr.edu.pl}

\author[N.\ Sandri\'{c}]{Nikola Sandri\'{c}}
\address[Nikola\ Sandri\'{c}]{
	Department of Mathematics\\University of Zagreb\\ Zagreb\\Croatia
}
\email{nsandric@math.hr}

\author[S.\ \v{S}ebek]{Stjepan\ \v{S}ebek}
\address[Stjepan\ \v{S}ebek]{
	Institute of Discrete Mathematics\\
	Graz University of Technology\\
	Graz\\ 
	Austria
	\&	
		Department of Applied Mathematics\\
	Faculty of Electrical Engineering and Computing\\
	University of Zagreb\\ 
 Zagreb\\ 
	Croatia}
\email{stjepan.sebek@fer.hr}


\begin{document}
\allowdisplaybreaks[4]

\begin{abstract}
In this article, we study fluctuations of the volume of a stable
sausage defined via a $d$-dimensional rotationally invariant
$\alpha$-stable process.
As the main results, we establish a functional central limit theorem (in the case when $d/\alpha>3 /2$)
with a standard one-dimensional Brownian motion in the limit, and  Khintchine's  and Chung's
laws of the iterated logarithm (in the case when $d/\alpha>9 /5$).

\end{abstract}
\maketitle

\section{Introduction}
 
Let $\X = \{X_t\}_{t \ge 0}$ be a L\'evy process in $\bbR^d$ defined on a probability space $(\Omega, \calF, \bbP)$.  A L\'{e}vy   sausage  associated with the process $\X$ and a given compact set $K\subset\bbR^d$, on the time interval $[s, t]$, $0 \le s \le t $,  is the random set  defined as
\begin{equation*}
	\calS^{K}[s, t] \,=\, \bigcup_{s \le u \le t} \{X_u + K\} .
\end{equation*}
If $s=0$ we use the notation $\calS^K_t = \calS^K[0, t]$.
 Let $\lambda(\D x)$ be the Lebesgue measure on $\bbR^d$ and let us denote by
\begin{equation*}
	\calV^K[s, t] \,=\, \lambda(\calS^K[s, t]) 
\end{equation*}
the volume of the L\'evy sausage $\calS^K[s, t]$ (we write $\calV^K_t= \lambda(\calS^K_t)$). Already Spitzer \cite{Spitzer} linked $\calV^K_t$ with the first hitting time $\tau_K = \inf \{s \ge 0 : X_s \in K\}$ via the identity
\begin{equation}\label{eq:EXP}
	\bbE[\calV_t^K] \,=\, \int_{\bbR^d} \bbP_x(\tau_{K} \le t)\,\D x,\qquad t\ge0,
\end{equation}
where $\bbP_x$ is the probability measure related to the process $\X$ started at $x\in \bbR^d$.
Port and Stone \cite[Theorem 11.1]{Port_Stone_Fourier} proved that if $\X$ is transient then 
 \begin{equation}\label{eq:SLLN}
\lim_{t \nearrow \infty}\frac{\bbE[\calV_t^K]}{t}\,=\,\Capa(K),
\end{equation}
where $\Capa(K)$  is the capacity of $K$ associated with the process $\X$.  Hawkes  \cite{Hawkes} observed that in view of the subadditivity of the process $\{\calV^K_t\}_{t\ge0}$, that is,
\begin{align*}
\calV^K_{s+t}\,\le\, \calV^K_s+\calV^K[s,s+t],\qquad s,t\ge 0,
\end{align*}
\cref{eq:SLLN} combined with Kingman's ergodic theorem (cf.\ \cite[Theorem Ch.\ I, 5.6]{Krengel}) and \cite[Proposition 3.12]{Kallenberg} implies the following strong law of large numbers
\begin{equation}\label{eq:SLLN2}\lim_{t \nearrow \infty}\frac{\calV^K_t}{t}\,=\,\Capa(K)\qquad \Prob\text{-a.s.}
\end{equation}

More satisfactory limit theorems for the volume of a L\'{e}vy sausage are known if $\X$ is a standard Brownian motion. In this case $\calS^K_t$ is called a Wiener sausage, and there is a vast amount of literature concerning its asymptotic behavior. The pioneering work \cite{Donsker_Varadhan} was due to Donsker and Varadhan were they established a large deviation principle for the volume of a Wiener sausage. Their result was extended  by Eisele and Lang \cite{Eisele} to the case when the driving process is a standard Brownian motion with drift, and to a class of elliptic diffusions by Sznitman \cite{Sznitman}, while \^{O}kura investigated similar questions for a certain class of symmetric L\'{e}vy processes.
Le Gall \cite{Le-Gall} obtained a central limit theorem for the volume of a Wiener sausage in dimensions $d\geq 2$, with different normalizing sequences and distributions in the limit for $d=2$, $d=3$ and $d\geq 4$, respectively. More recently, van den Berg, Bolthausen and den Hollander \cite{Berg_Bolthausen_Hollander} studied the problem of intersections of two Wiener sausages, see also \cite{van_den_Berg}, \cite{van_den_Berg_3} and \cite{van_den_Berg_2}. For further limit theorems for the volume of a Wiener sausage see \cite{Csaki}, \cite{Peres} and \cite{Wang-Gao}. 
We  remark that first studies on a Wiener sausage were motivated by its applications in physics \cite{Kac}.
We refer the reader to the book by Simon \cite{Simon} for a comprehensive discussion on this topic.

In the present article, we focus on the limit behavior of the volume of a stable sausage, that is, a L\'evy sausage corresponding to a 
stable L\'{e}vy process. Asymptotic behavior of stable sausages has not been extensively studied yet. 
In the seminal paper \cite{Donsker_Varadhan} Donsker and Varadhan obtained a large deviation principle for the volume of a stable sausage.
Some other works were concerned with the expansion of the expected volume of a stable sausage. More precisely, 
Getoor \cite{Getoor} proved \cref{eq:SLLN} for rotationally invariant $\alpha$-stable processes with $d>\alpha$ and for any compact set $K$. He also investigated the first order asymptotics of the difference $\bbE[\calV^K_t] - t\Capa(K)$, whose form depends on the value of the ratio $d/\alpha$, see \cite[Theorem 2]{Getoor}. The second order terms in this expansion were found by Port  \cite{Port} for all strictly stable processes satisfying some extra assumptions. 
In \cite{Rosen_sausage_plane} Rosen established asymptotic expansions for the volume of a  stable sausage in the plane with the coefficients represented by $n$-fold self-intersections of the stable process.
In this article, we obtain a central limit theorem  for the  volume of a stable sausage. We then apply this result  to study convergence of the volume process in the Skorohod space, and establish the corresponding functional central limit theorem. 
Finally, we also obtain Khintchine's  and Chung's laws of the iterated logarithm for this process.

Before we formulate our results, we briefly recall some basic notation from the potential theory of stable processes. Let $\X$ be a 
rotationally invariant stable L\'{e}vy process of index $\alpha \in (0,2]$, that is, a L\'{e}vy process whose bounded continuous transition density $p(t, x)$ is uniquely determined by the Fourier transform
\begin{align*}
\E^{-t|\xi|^\alpha} \, =\, \int_{\bbR^d}\E^{i(x,\xi)}\,p(t,x)\, \D x,
\end{align*}
where $(x,\xi)$ stands for the inner product in $\bbR^d$, $|x|=(x,x)^{1/2}$ is the Euclidean norm, and $\D x =\lambda (\D x)$.
We assume that $\X$ is transient, which holds if (and only if) $d>\alpha$. Its Green function is then given by  $G(x)= \int_0^{\infty} p(t, x)\,\D t$. Let $\mathfrak{B}(\bbR^d)$ denote the family of all Borel subsets of $\bbR^d$. 
For each $B\in\mathfrak{B}(\bbR^d)$ there exists a unique Borel measure $\mu_B(\D x)$ supported on $B\in\mathfrak{B}(\bbR^d)$ such that
	\begin{equation}\label{eq:G*mu=P(T<infty)}
		\bbP_x(\tau_B < \infty) \,=\, \int_{\bbR^d} G(x- y) \,\mu_B(\D y).
	\end{equation}
The measure $\mu_B(\D x)$ is called the equilibrium measure of $B$, and its capacity $\Capa(B)$ is defined as the total mass of $\mu_B(\D x)$, that is, $\Capa(B) = \mu_B (B)$. We denote by $\sB(x,r)$ the closed Euclidean ball centered at $x\in\R^d$  of radius $r>0$. In the case when $r=1$ and $x=0$, we write $\sB=\sB(0,1)$. If $B =\sB(0,r)$ then the measure $\mu_B (\D y)$ has a density which is proportional to $(r-|y|^2)^{-\alpha /2}$. 
In particular, we have (see for instance \cite{Takeuchi})
\begin{align*}
\Capa (\sB) = \frac{\Gamma (d/2)}{\Gamma (\alpha /2)\Gamma (1+(d-\alpha)/2)}.
\end{align*}
In the case when $K=\sB$, we simply write $\calV_t$ instead of $\calV^{\sB}_t$ (and similarly $\calS_t$ for $\calS^\sB_t$). Let $\calN(0, 1)$ denote the Gaussian random variable with mean zero and variance one.
Our central limit theorem (see \Cref{tm:CLT}) for the volume of a stable sausage asserts that if $d/\alpha> 3/2$ then
	there exists a constant $\sigma =\sigma(d,\alpha)> 0$  such that
	\begin{equation}\label{Result:CLT}
		\frac{\calV_t - t\Capa(\sB)}{\sigma\sqrt{t}}\, \xrightarrow[t\nearrow\infty]{({\rm d})} \,\calN(0, 1),
	\end{equation}
where convergence holds in distribution. 
The cornerstone of the proof of \cref{Result:CLT} is to represent  $\calV_t$ as a sum of independent random variables plus an error term. For this we use inclusion-exclusion formula together with the Markov property and rotational invariance of the process $\X$. More precisely, for $t, s \ge 0$, we have 
\begin{equation}
\begin{aligned}\label{eq:VOL}
	\calV_{t + s}
	& \,=\, \lambda\bigl(\calS_t \cup \calS[t, t + s]\bigr) \,=\, \lambda\bigl((\calS_t - X_t) \cup (\calS[t, t + s] - X_t)\bigr) \\
	& \,=\, \calV_t^{(1)} + \calV_s^{(2)} - \lambda\bigl(\calS_t^{(1)} \cap \calS_s^{(2)}\bigr),
\end{aligned}
\end{equation}
where $\calV_t^{(1)}$ and $\calV_s^{(2)}$ ($\calS_t^{(1)}$ and $\calS_s^{(2)}$) are independent and have the same law as $\calV_t$ and $\calV_s$ ($\calS_t$ and $\calS_s$), respectively.
This decomposition allows us to apply the Lindeberg-Feller central limit theorem in the present context. 
The first key step is to find estimates for the error term $\lambda\bigl(\calS_t^{(1)} \cap \calS_s^{(2)}\bigr)$, which we give in  \Cref{Subsec:error}.
The second step is to control the variance of the volume of a stable sausage which is achieved in  \Cref{Subsec:variance}.

Let us emphasize that the present article has been mainly inspired by Le Gall's work \cite{Le-Gall}  where he studied fluctuations of the volume of a Wiener sausage (the case  $\alpha =2$).  
Among other results, he established the central limit theorem in \cref{Result:CLT} for dimensions $d\geq 4$. Still another source of motivation was the article \cite{LeGall-Rosen} by Le Gall and Rosen where they proved a corresponding central limit theorem for the range of stable random walks and mentioned that it is plausible that similar result holds for stable sausages, see \cite[Page 654]{LeGall-Rosen}. Both of these articles were also concerned with the lower-dimensional case $d<4$ and $d/\alpha \leq 3/2$, respectively. In the present article we are only interested in the case when $d/\alpha >3 /2$, and we postpone the study of the remaining values of the ratio $d/\alpha$ to follow-up articles. 

As an application of  \cref{Result:CLT} we obtain a functional central limit theorem (see \Cref{tm:FCLT}) which states that under the same  assumptions, and with the same constant $\sigma >0$, 
	\begin{equation}\label{Result:FCLT}
		\VIT{\frac{\calV_{nt} - nt\Capa(\sB)}{\sigma \sqrt{n}}}_{t \ge 0}\, \xrightarrow[n\nearrow\infty]{({\rm J}_1)}\, \{W_t\}_{t \ge 0}.
	\end{equation}
Here, convergence holds in the Skorohod space $\calD([0, \infty), \bbR)$ endowed with the ${\rm J}_1$ topology, and $\{W_t\}_{t \ge 0}$ is a standard Brownian motion in $\bbR$. The proof of \cref{Result:FCLT} is performed according to a general two-step scheme: (i) convergence of finite-dimensional distributions, which follows from \cref{Result:CLT}; (ii) tightness, which we investigate  by employing the well-known Aldous criterion, see  \Cref{Sec:FCLT} for details.

It is remarkable that results in \cref{Result:CLT,Result:FCLT} correspond  to
analogous results for the range (and its capacity) of stable
random walks on the integer lattice $\bbZ^d$ which we discussed in
\cite{CSS19} and \cite{CSS19_2}, respectively.

We finally use \cref{Result:CLT} to study growth of the paths of the volume of the stable sausage. In the case when $d/\alpha>9/5$, we prove   Khintchine's  law of the iterated logarithm
	\begin{align}\label{Result:LIL(Khi)}
		\liminf_{t \nearrow \infty} \frac{\calV_t - t\Capa(\sB)}{\sqrt{2\sigma^2 t \log \log t}} \,=\, -1\quad \mathrm{and}\quad
		\limsup_{t \nearrow \infty}\frac{\calV_t - t\Capa(\sB)}{\sqrt{2\sigma^2 t \log \log t}} \,=\, 1\qquad \Prob\text{-a.s.},
	\end{align} as well as
 Chung's law of the iterated logarithm
	\begin{align}\label{Result:LIL(Chung)}
		\liminf_{t \nearrow \infty} \frac{\sup_{0 \le s \le t} \aps{\calV_s - s\Capa(\sB)}}{\sqrt{\sigma^2 t/\log \log t}} \,=\, \frac{\pi}{\sqrt{8}}\qquad \Prob\text{-a.s.}
	\end{align} 
Our proof is based on the approach which was developed by Jain and Pruitt \cite{Jain-Pruitt} (see also \cite{Bass-Kumagai}) in the context of the range of random walks and then successfully applied in \cite{Wang-Gao} to obtain Khintchine's  and Chung's laws of the iterated logarithm for a Wiener sausage in dimensions $d\geq 4$.

The rest of the paper is organized as follows. In  \Cref{Sec:CLT} we prove the central limit theorem in \cref{Result:CLT}. For this we first deal with the error terms derived from \cref{eq:VOL}, and in the second part we show that the variance of the volume of a stable sausage behaves linearly at infinity.  \Cref{Sec:FCLT} is devoted to the proof of \cref{Result:FCLT}, and in  \Cref{Sec:asIP} we prove \cref{Result:LIL(Khi),Result:LIL(Chung)}. In \Cref{Sec:Technical} we present the proofs of some technical results which we need in the course of the study.

\section{Central limit theorem}\label{Sec:CLT}
The goal of this section is to prove the following central limit theorem. We assume  that $\X$ is a  rotationally invariant stable
L\'{e}vy process in $\bbR^d$ of index $\alpha \in (0,2]$ satisfying $d/\alpha> 3/2$. 
\begin{theorem}\label{tm:CLT} 
Under the above assumptions, 
	there exists a constant $\sigma = \sigma(d,\alpha)> 0$  such that
	\begin{equation*}
		\frac{\calV_t - t\Capa(\sB)}{\sigma\sqrt{t}}\, \xrightarrow[t\nearrow\infty]{({\rm d})} \,\calN(0, 1).
	\end{equation*}
\end{theorem}
We remark that \Cref{tm:CLT} holds for any closed ball $\sB(x,r)$, with a possibly different constant $\sigma>0$. Moreover, as indicated by \cref{eq:SLLN}, the statement of the theorem remains valid if we replace the term $t\Capa(\sB)$ with $\bbE[\calV_t]$.

Before we embark on the proof of the theorem, which is given at the end of the section, we first need to find satisfactory estimates for the error term in decomposition \cref{eq:VOL}. Next step is to investigate the variance  $\Var(\calV_t) $ of the volume of a stable sausage and to show that it behaves as $\sigma t$ at infinity.

\subsection{Error term estimates}\label{Subsec:error}

We assume that $\X$ is defined on the canonical space $\Omega= \calD([0, \infty), \bbR^d)$ of all c\`{a}dl\`{a}g functions $\omega : [0, \infty) \to \bbR^d$. It is endowed with the Borel $\sigma$-algebra $\calF$ generated by the Skorokhod $\rm{J}_1$ topology. Then $\X$ is understood as the coordinate process, that is, $X_t(\omega) = \omega (t)$, and the shift operator $\theta_t$ acting on $\Omega$ is defined by
\begin{equation*}
	\theta_t \,\omega(s) \,=\, \omega(t + s), \qquad t,s \ge 0.
\end{equation*}
In what follows we use notation
\begin{equation*}
	\calS^K[t, \infty) \,=\, \bigcup_{s\ge t } \{X_s + K\},\qquad  t\ge 0.
\end{equation*}
We also write $\calS^K_\infty = \calS^K[0, \infty)$, $\calV^K_{\infty} = \lambda(\calS^K_{\infty})$ and $\calV^K[t, \infty) = \lambda(\calS^K[t, \infty))$, $t\geq 0$.
We start with a lemma which enables us to represent the expected volume of the intersection of two  sausages in terms of the difference $\bbE[\calV^K_t] - t\Capa(K)$.

\begin{lemma}\label{lm:obtaining_tC(B)}
	For any compact set $K$ and all $t \ge 0$ it holds 
	\begin{equation*}
		\bbE[\calV^K_t] - t\Capa(K) \,=\, \bbE[\lambda(\calS^K_t \cap \calS^K[t, \infty))].
	\end{equation*}
\end{lemma}
\begin{proof}
We clearly have 
	\begin{equation*}
		\calV^K_{\infty} \,=\, \calV^K_t + \calV^K[t, \infty) - \lambda(\calS^K_t \cap \calS^K[t, \infty))
	\end{equation*}
which implies
	\begin{equation*}
		\lambda(\calS^K_t \cap \calS^K[t, \infty)) \,=\, \calV^K_t - \lambda(\calS^K_t \setminus \calS^K[t, \infty)).
	\end{equation*}
	Hence, it suffices to show that
	\begin{equation}\label{eq:LM2.1A}
		\bbE[\lambda(\calS^K_t \setminus \calS^K[t, \infty)] \,=\, t\Capa(K),\qquad t\ge0.
	\end{equation}
	By rotational invariance of the process $\X$ we have	
	\begin{align}\label{al:obtaining_tC(B)}
		\bbE[\lambda(\calS^K_t \setminus \calS^K[t, \infty)]
		& \,=\, \int_{\bbR^d} \bbP(x \in \calS^K_t \setminus \calS^K[t, \infty))\,\D x \nonumber \\
		& \,=\, \int_{\bbR^d} \bbP_x \Big(\bigcup_{0\leq s\leq t}\{X_s \in K\},\, \bigcap_{s\geq t}\{X_s\notin K\}\Big)\,\D x\nonumber \\
		& \,=\, \int_{\bbR^d} \bbP_x(0 < \eta_{K} \le t)\,\D x,
	\end{align}
	where $\eta_{K}$ is the last exit time of the process $\X$ from  the set $K$, that is,
	\begin{equation*}
		\eta_{K} \,=\, \begin{cases}
			\sup\{t > 0 : X_t \in K\}, & \tau_{K} < \infty, \\
			0, & \tau_{K} = \infty.
		\end{cases}
	\end{equation*}
We observe that $\{\eta_K > t\} = \{\tau_K \circ \theta_t < \infty\}$ which together with \cref{eq:G*mu=P(T<infty)} yields
	\begin{align*}
		\bbP_x(\eta_{K} > t)
		& \,=\, \int_{\bbR^d} p(t, x, y) \,\bbP_y(\tau_{K} < \infty) \,\D y \,=\, \int_{\bbR^d} \int_t^{\infty} p(s, x, z)\, \D s\,\mu_{K}(\D z) ,
	\end{align*}
	where we used notation $p(t,x,y) = p(t,y-x)$.
We obtain 
	\begin{equation*}
		\bbP_x(0 < \eta_{K} \le t) \,=\, \int_{\bbR^d}  \int_0^t p(s, x, y)\,\D s\,\mu_{K}(\D y).
	\end{equation*}
This and  \cref{al:obtaining_tC(B)} imply
	\begin{equation*}
		\bbE[\lambda(\calS^K_t \setminus \calS^K[t, \infty))] \,=\, \int_{\bbR^d}  \int_0^t  \int_{\bbR^d} p(s, y, x)\, \D x\,\D s\,\mu_{K}(\D y) \,=\, t\Capa(K),
	\end{equation*}
and the proof is finished.
\end{proof}

In the following lemma we show how one can easily estimate the higher moments of the expected volume of the intersection of two sausages through the first moment estimate.
\begin{lemma}\label{LM:2.2}
Let $\X^\prime$ be an independent copy of the process $\X$ such that $X_0=X^\prime_0$, and let $\calS'_t$, $t\ge0$, denote the sausage associated with $\X^\prime$. Then 
	for all $k \in \bbN$ and $t \ge 0$ it holds 
	\begin{equation*}
		\bbE[\lambda(\calS_t \cap \calS'_{\infty})^k] \,\le\, 2^{k-1}(k!)^2\, \big(\bbE[\lambda(\calS_t \cap \calS'_{\infty})]\big)^k.
	\end{equation*}
\end{lemma}
\begin{proof}
	We observe that
	\begin{equation}\label{al:expression_for_the_error_term-1st_moment}
	\bbE[\lambda(\calS_t \cap \calS'_{\infty})]
	  \,=\ \int_{\bbR^d} \bbP(x \in \calS_t)\, \bbP(x \in \calS'_{\infty})\,\D x  
	\,=\, \int_{\bbR^d} \bbP_x(\tau_{\sB} \le t)\, \bbP_x(\tau_{\sB} < \infty)\, \D x,
	\end{equation}	
	where we used  rotational invariance of  $\X$. Similarly, for $k\ge1$ we have 
	\begin{equation}\label{al:k-th_moment-1st_part}
\bbE[\lambda(\calS_t \cap \calS'_{\infty})^k]
 \,=\, \int_{\bbR^d} \cdots \int_{\bbR^d} \bbP(x_1, \ldots, x_k \in \calS_t)\, \bbP(x_1, \ldots, x_k \in \calS_{\infty})\,\D x_1\cdots \D x_k.
	\end{equation}
By the strong Markov property, we obtain
	\begin{equation}
	\begin{aligned}\label{al:x1...xk_in_S-iteration-1st_part}
\bbP(x_1, \ldots, x_k \in \calS_t)
& \,=\, \bbP(\tau_{\sB - x_1} \le t, \ldots, \tau_{\sB-x_k} \le t) \\&\,\le\, k! \,\bbP(\tau_{\sB-x_1} \le \cdots \le \tau_{\sB-x_k} \le t) \\
& \,\le\, k!\, \bbP(\tau_{\sB-x_1} \le \cdots \le \tau_{\sB-x_{k - 1}} \le t)\!\! \sup_{z \in \sB-x_{k - 1}} \bbP_z(\tau_{\sB-x_k} \le t).
\end{aligned}
	\end{equation}
For any $w \in \sB$ we have $\sB - w \subseteq \sB(0, 2)$ and whence
	\begin{equation*}
		\bbP_{w - x_{k - 1}} (\tau_{\sB-x_k} \le t) \,=\, \bbP_{x_k - x_{k - 1}} (\tau_{\sB - w} \le t) \,\le\, \bbP_{x_k - x_{k - 1}} (\tau_{\sB(0, 2)}\le t).
	\end{equation*} 
	For $x\in\bbR^d$ and $B\in\mathfrak{B}(\bbR^d)$ we set
 $\tau_B^x = \inf\{t \ge 0 : x + X_t \in B\}$ and show that $\tau_{\sB(0, 2)}^x \overset{({\rm d})}{=} 2^\alpha\tau_{\sB}^{x/2}$, that is, the random variables $\tau_{\sB(0, 2)}^x $ and $2^\alpha\tau_{\sB}^{x/2}$ are equal in distribution. Indeed, the easy calculation yields
	\begin{align*}
		\tau_{\ \sB(0, 2)}^x
		& \,=\, \inf\{t \ge 0 : \aps{x + X_t} \le 2\} \,=\, \inf\VIT{t \ge 0 : \APS{\frac{x}{2} + \frac{X_t}{2} } \le 1} \\
		& \,\overset{({\rm d})}{=}\, \inf\VIT{t \ge 0 : \APS{\frac{x}{2} + X_{t/2^{\alpha}}} \le 1} \,=\, 2^{\alpha} \inf\VIT{s \ge 0 : \APS{\frac{x}{2} + X_s} \le 1} \,=\, 2^{\alpha} \tau_{\sB}^{x/2}.
	\end{align*}
	This implies that for arbitrary $w \in \sB$,
	\begin{equation*}
		\bbP_{w - x_{k - 1}}(\tau_{\sB-x_k} \le t) \,\le\, \bbP(\tau_{\sB(0, 2)}^{x_k - x_{k - 1}} \le t)\, =\, \bbP(2^{\alpha}\tau_{\sB}^{(x_k - x_{k - 1})/2} \le t) \,\le\, \bbP_{\frac{x_k - x_{k - 1}}{2}} (\tau_{\sB} \le t).
	\end{equation*}
	In particular,
	\begin{equation*}
		\sup_{z \in \sB - x_{k - 1}} \bbP_z(\tau_{\sB-x_k} \le t) \,\le\, \bbP_{\frac{x_k - x_{k - 1}}{2}} (\tau_{\sB} \le t).
	\end{equation*}
	By combining this with \cref{al:x1...xk_in_S-iteration-1st_part} and iterating the whole procedure, we obtain
	\begin{equation*}
		\bbP(x_1, \ldots, x_k \in \calS_t)\, \le\, k!\, \bbP_{x_1}(\tau_{\sB}  \le t) \,\bbP_{\frac{x_2 - x_1}{2}} (\tau_{\sB}  \le t) \cdots \bbP_{\frac{x_k - x_{k - 1}}{2}} (\tau_{\sB} \le t).
	\end{equation*}
	Similarly, it follows that
	\begin{equation*}
		\bbP(x_1, \ldots, x_k \in \calS_{\infty}) \,\le\, k!\, \bbP_{x_1}(\tau_{\sB} < \infty)\, \bbP_{\frac{x_2 - x_1}{2}} (\tau_{\sB} < \infty) \cdots \bbP_{\frac{x_k - x_{k - 1}}{2}} (\tau_{\sB}  < \infty).
	\end{equation*}
	Applying the last two inequalities to \cref{al:k-th_moment-1st_part} and using  \cref{al:expression_for_the_error_term-1st_moment} finishes the proof.
\end{proof}

\begin{corollary}\label{Cor:moment_bounds}
In the notation of  \Cref{LM:2.2}, 
for all $k\in \bbN$ and $t>0$ large enough there is a constant $c=c(d,\alpha)>0$ such that
\begin{equation}\label{eq:k-th_moment_bound}
	\bbE[\lambda(\calS_t \cap \calS'_{\infty})^k]  \,\le\, 2^{k-1}(k!)^2c^k \,h(t)^k,
\end{equation} 
where the function $h:(0,\infty)\to (0,\infty)$ is defined as 
\begin{equation}\label{eq:def_of_hd}
	h(t) \,=\,
	\begin{cases}
		1, & d/\alpha > 2, \\
		\log(t+\E), & d/\alpha = 2, \\
		t^{2 - d/\alpha}, & d/\alpha \in (1, 2).
	\end{cases}
\end{equation}
\end{corollary}

\begin{proof}
It follows from \cite[Theorem 2]{Getoor} that there is a constant $c(d,\alpha) > 0$ such that for $t>0$ large enough
\begin{equation}\label{eq:Getoor}
	\int_{\bbR^d} \bbP_x(\tau_{\sB} \le t)\,\D x - t\Capa(\sB) \,\le\, c(d,\alpha) \,h(t),
\end{equation}
where the function $h(t)$ is given in \cref{eq:def_of_hd}.
We observe that  by the Markov property and translation invariance of $\lambda (\D x)$ we have
\begin{align*}
\bbE[\lambda(\calS_t \cap \calS'_{\infty})]&\,=\, \bbE[\lambda((\calS_t - X_t) \cap (\calS[t, \infty) - X_t))]\,=\,\bbE[\lambda(\calS_t \cap \calS[t, \infty))].
\end{align*}
Thus, \cref{eq:EXP} combined with \Cref{lm:obtaining_tC(B)} 
and \cref{eq:Getoor}  implies the assertion for $k=1$.
For $k>1$ the result follows by \Cref{LM:2.2}.
\end{proof}

\subsection{Variance of the volume of a stable sausage}\label{Subsec:variance}

Our aim in this section is to determine the constant $\sigma$ in  \Cref{tm:CLT}. We can easily adapt the approach of \cite[Lemma 4.3]{CSS19} to the present setting and combine it with \cite[Theorem 2]{Hammersley} to conclude that the limit below exists
\begin{equation}\label{eq:VAR}
	\lim_{t \nearrow \infty} \frac{\Var(\calV_t)}{t} \,=\,\sigma^2.
\end{equation} 
The main difficulty is to show that $\sigma$ is strictly positive, and this is obtained in the following crucial lemma. We adapt the proof of \cite[Lemma 4.2]{Le-Gall} but let us emphasize that it is a laborious task to adjust it to the case of stable processes.

\begin{lemma}\label{lm:positivity}
	The constant $\sigma$ in \cref{eq:VAR} is strictly positive.
\end{lemma}
\begin{proof}  
We split the proof into several steps and we notice that it clearly suffices to restrict our attention to integer values of the parameter $t$.
 \medskip
 
\noindent 
	\textit{Step 1}. We start by finding a handy decomposition of  the variance  $\Var(\calV_n)$ expressed as a sum of specific random variables, see \cref{eq:lim_Var=lim_Y}.
	 We assume that $X_0 = 0$, and  we set
	\begin{equation}\label{eq:def_of_S-hat}
	\widehat{\calS}[s, t] \,=\, \calS[s, t] \setminus \calS_s,\qquad 0 \le s \le t<\infty.
	\end{equation}
	For $n, N \in \bbN$ such that $1 \le n \le N$ we have
		\begin{equation*}
		\calV_n + \lambda(\widehat{\calS}[n, N]) \,=\, \calV_N.
	\end{equation*}
Let $\calF_t = \sigma(X_s : s \le t)$. Since $\calV_n$ is $\calF_n$-measurable, we obtain
	\begin{equation*}
		\calV_n + \bbE[\lambda(\widehat{\calS}[n, N]) \mid \calF_n] \,=\, \bbE[\calV_N \mid \calF_n]
	\end{equation*}
and by taking expectations and subtracting\footnote{For any random variable $Y\in {\rm L}^1$ we write $\langle Y \rangle = Y - \bbE[Y]$.}
	\begin{equation*}
		\langle\calV_n\rangle + \langle\bbE[\lambda(\widehat{\calS}[n, N]) \mid \calF_n]\rangle \,=\, \bbE[\calV_N \mid \calF_n] - \bbE[\calV_N].
	\end{equation*}
	Hence
	\begin{equation}\label{eq:key_relation-1st_iteration}
		\langle\calV_n\rangle + \langle\bbE[\lambda(\widehat{\calS}[n, N]) \mid \calF_n]\rangle \,=\, \sum_{k = 1}^n U_k^N,
	\end{equation} 
	where
	\begin{equation*}
		U_k^N \,=\, \bbE[\calV_N \mid \calF_k] - \bbE[\calV_N \mid \calF_{k - 1}].
	\end{equation*}
We first discuss the second term on the left-hand side of \cref{eq:key_relation-1st_iteration}.
We claim that 
	\begin{equation}\label{eq:key_relation-aux}
	\langle\bbE[\lambda(\widehat{\calS}[n, N]) \mid \calF_n]\rangle \,=\, -\langle\bbE[\lambda(\calS[n, N] \cap \calS_n) \mid \calF_n]\rangle.
	\end{equation}
	 Indeed, by \cref{eq:def_of_S-hat} we have
	\begin{equation*}
		\lambda(\widehat{\calS}[n, N]) \,=\, \lambda(\calS[n, N]) - \lambda(\calS[n, N] \cap \calS_n)
	\end{equation*}
	and the independence of the increments of the process $\X$ implies that
	\begin{align*}
		\bbE[\lambda(\widehat{\calS}[n, N]) \mid \calF_n] &\,=\ \bbE[\lambda(\calS[n, N])|\calF_n] - \bbE[\lambda(\calS[n, N] \cap \calS_n) \mid \calF_n]\\
		&\,=\,\bbE[\lambda(\calS[n, N]-X_n)|\calF_n] - \bbE[\lambda(\calS[n, N] \cap \calS_n) \mid \calF_n]\\
		&\,=\,\bbE[\lambda(\calS[n, N]-X_n)] - \bbE[\lambda(\calS[n, N] \cap \calS_n) \mid \calF_n]
		\\
		&\,=\,\bbE[\lambda(\calS[n, N])] - \bbE[\lambda(\calS[n, N] \cap \calS_n) \mid \calF_n].
	\end{align*}
	Taking expectation  and then subtracting the two relations yields \cref{eq:key_relation-aux}.

	Next we deal with the random variables $U_k^N$ for $k=1,\ldots ,N$. By the independence of the increments of the process $\X$, we obtain
	\begin{align*}
		U_k^N
		&\, =\, \bbE[\lambda(\calS_k) + \lambda(\calS[k, N]) - \lambda(\calS_k \cap \calS[k, N]) \mid \calF_k] \\
		& \ \ \ \ - \bbE[\lambda(\calS_{k - 1}) + \lambda(\calS[k - 1, N - 1]) - \lambda(\calS_{k - 1} \cap \calS[k - 1, N - 1])  \mid \calF_{k - 1}] \\
		&\ \ \ \   -\bbE[ \lambda(\widehat{\calS}[N - 1, N]) \mid \calF_{k - 1}] \\
		& \,=\, \lambda(\calS_k) - \lambda(\calS_{k - 1}) + \bbE[\lambda(\calS[k, N])] - \bbE[\lambda(\calS[k - 1, N - 1])] \\
		&\ \ \ \ - \bbE[\lambda(\widehat{\calS}[N - 1, N]) \mid \calF_{k - 1}]  + \bbE[\lambda(\calS_{k - 1} \cap \calS[k - 1, N - 1]) \mid \calF_{k - 1}] \\ 
		&\ \ \ \  - \bbE[\lambda(\calS_k \cap \calS[k, N]) \mid \calF_k] \\
		& \,=\, \lambda(\widehat{\calS}[k - 1, k]) - \bbE[\lambda(\widehat{\calS}[N - 1, N]) \mid \calF_{k - 1}] + \bbE[\lambda(\calS_{k - 1} \cap \calS[k - 1, N - 1]) \mid \calF_{k - 1}] \\&\ \ \ \ - \bbE[\lambda(\calS_k \cap \calS[k, N]) \mid \calF_k].
	\end{align*}
	Let $\calF_{s, t}$ denote the $\sigma$-algebra generated by the increments of $\X$ on $[s, t]$, $0\le s\le t$. Then, by a  reversibility argument, 
	\begin{equation*}
		\bbE[\lambda(\widehat{\calS}[N - 1, N]) \mid \calF_{k - 1}]\, \overset{({\rm d})}{=}\, \bbE[\lambda(\calS_1 \setminus \calS[1, N]) \mid \calF_{N - k + 1, N}].
	\end{equation*}
	Moreover, the following convergence in ${\rm L}^1$ holds
		\begin{equation}\label{Lemma_technical_1}
		\bbE[\lambda(\calS_1 \setminus \calS[1, N]) \mid \calF_{N - k + 1, N}] \,\xrightarrow[N\nearrow\infty]{{\rm L}^1}\, \bbE[\lambda(\calS_1 \setminus \calS[1, \infty))].
	\end{equation}
	The proof of \cref{Lemma_technical_1} is postponed to \Cref{Sec:Technical}, \Cref{lm:L1-cvg_with_Levi}.  
	In view of \cref{eq:LM2.1A} it follows that
	\begin{equation*}
		\bbE[\lambda(\widehat{\calS}[N - 1, N]) \mid \calF_{k - 1}] \,\xrightarrow[N\nearrow\infty]{{\rm L}^1}\, \Capa(\sB).
	\end{equation*}
We thus obtain
	\begin{equation}
	\begin{aligned}\label{eq:lim_UkN-1st_part}
			 U_k^N
			\xrightarrow[N\nearrow\infty]{{\rm L}^1}\,& \lambda(\widehat{\calS}[k - 1, k]) - \Capa(\sB) + \bbE[\lambda(\calS_{k - 1} \cap \calS[k - 1, \infty)) \mid \calF_{k - 1}] \\
			&- \bbE[\lambda(\calS_k \cap \calS[k, \infty)) \mid \calF_k].
	\end{aligned}
	\end{equation}
Further, we observe that
	\begin{align*}
		\lambda(\calS_k \cap \calS[k, \infty))
		& \,=\, \lambda(\calS_{k - 1} \cap \calS[k - 1, \infty)) - \lambda(\calS_{k - 1} \cap \calS[k - 1, k]) \\
		& \ \ \ \ + \lambda(\calS_k \cap \calS[k, \infty) \cap \calS_{k - 1} \cap \calS[k - 1, k]) + \lambda(\calS[k - 1, k] \cap \calS[k, \infty))\\
		& \ \ \ \  - \lambda(\calS_{k - 1} \cap \calS[k - 1, \infty) \cap \calS[k - 1, k] \cap \calS[k, \infty)) \\
		& \,=\, \lambda(\calS_{k - 1} \cap \calS[k - 1, \infty)) - \lambda(\calS_{k - 1} \cap \calS[k - 1, k])\\
		& \ \ \ \ + \lambda(\calS[k - 1, k] \cap \calS[k, \infty)).
	\end{align*}
This and \cref{eq:lim_UkN-1st_part} imply
	\begin{align*}
		 U_k^N
		 \,\xrightarrow[N\nearrow\infty]{{\rm L}^1}\,&\lambda(\calS[k - 1, k] \cap \calS_{k - 1}^c) - \Capa(\sB) + \bbE[\lambda(\calS_{k - 1} \cap \calS[k - 1, \infty)) \mid \calF_{k - 1}] \\
		&  \ \ \ \ - \bbE[\lambda(\calS_{k - 1} \cap \calS[k - 1, \infty)) \mid \calF_k] + \lambda(\calS[k - 1, k] \cap \calS_{k - 1}) \\
		& \ \ \ \  - \bbE[\lambda(\calS[k - 1, k] \cap \calS[k, \infty)) \mid \calF_k] \\
		& = \lambda(\calS[k - 1, k]) - \bbE[\lambda(\calS[k - 1, k] \cap \calS[k, \infty)) \mid \calF_k] - \Capa(\sB) \\
		& \ \ \ \  + \bbE[\lambda(\calS_{k - 1} \cap \calS[k - 1, \infty)) \mid \calF_{k - 1}] - \bbE[\lambda(\calS_{k - 1} \cap \calS[k - 1, \infty)) \mid \calF_k] \\
		& = \langle\bbE[\lambda(\calS[k - 1, k] \setminus \calS[k, \infty)) \mid \calF_k]\rangle + \bbE[\lambda(\calS[k - 1, k] \setminus \calS[k, \infty))] - \Capa(\sB) \\
		& \ \ \ \  + \bbE[\lambda(\calS_{k - 1} \cap \calS[k - 1, \infty)) \mid \calF_{k - 1}] - \bbE[\lambda(\calS_{k - 1} \cap \calS[k - 1, \infty)) \mid \calF_k] \\
		& = \langle\bbE[\lambda(\calS[k - 1, k] \setminus \calS[k, \infty)) \mid \calF_k]\rangle \\
		& \ \ \ \ + \bbE[\lambda(\calS_{k - 1} \cap \calS[k - 1, \infty)) \mid \calF_{k - 1}] - \bbE[\lambda(\calS_{k - 1} \cap \calS[k - 1, \infty)) \mid \calF_k],
	\end{align*}
	where in the last line we used \cref{eq:LM2.1A}.
	Hence, by \cref{eq:key_relation-1st_iteration,eq:key_relation-aux},
	\begin{equation}\label{eq:key_relation}
		\langle \calV_n \rangle \,=\, \langle\bbE[\lambda(\calS_n \cap \calS[n, \infty)) \mid \calF_n]\rangle + \sum_{k = 1}^n Y_k,
	\end{equation}
	where
	\begin{equation}
		\begin{aligned}\label{al:expr_for_Y_k}
			Y_k
			& \,=\, \bbE[\lambda(\calS_{k - 1} \cap \calS[k - 1, \infty)) \mid \calF_{k - 1}] - \bbE[\lambda(\calS_{k - 1} \cap \calS[k - 1, \infty)) \mid \calF_k] \\
			& \ \ \ \  + \langle\bbE[\lambda(\calS[k - 1, k] \setminus \calS[k, \infty)) \mid \calF_k]\rangle.
		\end{aligned}	
	\end{equation}
	From \cref{eq:key_relation} it follows that the variance of $\calV_{n}$ is equal to
	\begin{equation}
		\begin{aligned}\label{al:Var(V_n)}
			\Var(\calV_n)
			& \,=\, \Var(\bbE[\lambda(\calS_n \cap \calS[n, \infty)) \mid \calF_n]) 
			+ \bbE \Big[\Big(\sum_{k = 1}^n Y_k\Big)^2\Big]\\
			& \ \ \ \ + 2\,\bbE \Big[\langle\bbE[\lambda(\calS_n \cap \calS[n, \infty)) \mid \calF_n]\rangle \sum_{k = 1}^n Y_k\Big].
	\end{aligned}	
	\end{equation}
	Clearly, $Y_k$ is $\calF_k$-measurable and $\bbE[Y_l \mid \calF_k] = 0$ for $k < l$. It follows that
	\begin{align*}
		\bbE \Big[\Big(\sum_{k = 1}^n Y_k\Big)^2\Big]
		\,=\, \sum_{k = 1}^n \bbE[Y_k^2].
	\end{align*}
Jensen's inequality and \cref{eq:k-th_moment_bound} with $d > 3\alpha/2$ yield
	\begin{align*}
	\lim_{n \nearrow \infty} \frac{1}{n} \Var(\bbE[\lambda(\calS_n \cap \calS[n, \infty)) \mid \calF_n])
	& \,\le\, \lim_{n \nearrow \infty} \frac{1}{n} \bbE[\lambda(\calS_n \cap \calS[n, \infty))^2]  \,=\, 0.
	\end{align*}
	The sequence $\{\frac{1}{n} \sum_{k = 1}^n \bbE[Y_k^2]\}_{n \ge1}$ is bounded (the proof is given in \Cref{Sec:Technical}, \Cref{lm:bounded}), and thus by 
the Cauchy-Schwarz inequality we conclude that
	\begin{align*}
	&\lim_{n \nearrow \infty}\frac{1}{n}
	\bbE \Big[\langle\bbE[\lambda(\calS_n \cap \calS[n, \infty)) \mid \calF_n]\rangle \sum_{k = 1}^n Y_k\Big]\,=\,0.
	\end{align*}
We have shown that
	\begin{equation}\label{eq:lim_Var=lim_Y}
		\lim_{n \nearrow \infty} \frac{\Var(\calV_n)}{n} \,=\, \lim_{n \nearrow \infty} \frac{1}{n} \sum_{k = 1}^n \bbE[Y_k^2].
	\end{equation}
	
\noindent	
\textit{Step 2}. In this step we prove that the limit on the right-hand side of \cref{eq:lim_Var=lim_Y} is strictly positive.
Let $\X'$ be an independent copy of the original process $\X$ such that $\X'_0=0$ and it has c\`{a}gl\`{a}d paths. We consider a process $\bar\X=\{\bar X_t\}_{t\in\bbR}$ by setting $\bar X_t=X'_{-t}$ for $t\le0$, and $\bar X_t = X_{t}$ for $t\ge0$. Clearly, the process $\bar\X$ has c\`{a}dl\`{a}g paths, and stationary and independent increments. The sausages $\calS[s, t]$, $\calS(-\infty, s]$ and $\calS[s, \infty)$ corresponding to $\bar \X$ are defined for all $s, t \in \bbR$, $s \le t$. Recall that $\calS_\infty = \calS[0,\infty )$. 
We assume that the process $\bar\X$ is defined on the canonical space $\Omega = \calD(\bbR, \bbR^d)$ of all c\`{a}dl\`{a}g functions $\omega \colon \bbR \to \bbR^d$ as the coordinate process $\bar X_t(\omega) = \omega(t)$. We define a shift operator $\vartheta$ acting on $\Omega$ by
	\begin{equation}\label{eq:def_of_vartheta}
		\vartheta\, \omega(t) \,=\, \omega(1 + t) - \omega(1), \qquad t \in \bbR,
	\end{equation}
and we notice that it is a $\bbP$-preserving mapping. For $t \in \bbR$, we define
	\begin{equation*}
		\calG_t \,=\, \sigma(\bar{X}_s : -\infty < s \le t),
	\end{equation*}
and for $ k\in\N$, 
	\begin{equation}\label{eq:def_of_Zk}
		Z_k \,=\, \bbE[\lambda(\calS[-k, 0] \cap \calS_{\infty}) \mid \calG_0] - \bbE[\lambda(\calS[-k, 0] \cap \calS_{\infty}) \mid \calG_1] + \langle\bbE[\lambda(\calS_1 \setminus \calS[1, \infty)) \mid \calG_1]\rangle.
	\end{equation}
	By \cref{al:expr_for_Y_k,eq:def_of_vartheta} it follows that
	\begin{equation}\label{eq:pres}
		Y_k \,=\, Z_{k - 1} \circ \vartheta^{k - 1}.
	\end{equation}
	In the sequel, we prove  that there exists a random variable $Z$ such that $\bbE[|Z|]>0$ and 
	\begin{equation}\label{eq:Z_k->Z}
		Z_k \,\xrightarrow[k\nearrow\infty]{{\rm L}^1} \,Z.
	\end{equation}
	
\noindent
\textit{Step 2a}. We start by proving the existence of $Z$.	
	This is evident if $d > 2\alpha$ as in this case $h(t) = 1$ and whence using \cref{eq:k-th_moment_bound} we obtain
\begin{equation*}
	\bbE[\lambda(\calS(-\infty, 0] \cap \calS_{\infty}]\, <\, \infty.
\end{equation*}
This implies  \eqref{eq:Z_k->Z} with
\begin{equation}
\begin{aligned}\label{eq:shape_of_Z}
	Z \,=\, &\bbE[\lambda(\calS(-\infty, 0] \cap \calS_{\infty}) \mid \calG_0] - \bbE[\lambda(\calS(-\infty, 0] \cap \calS_{\infty}) \mid \calG_1] \\
	&+ \langle\bbE[\lambda(\calS_1 \setminus \calS[1, \infty)) \mid \calG_1]\rangle.
	\end{aligned}
\end{equation}
We next consider the case $3\alpha /2 < d \le 2\alpha$. Then $\bbE[\lambda(\calS(-\infty, 0] \cap \calS_{\infty}]$ is not finite and we cannot define $Z$ as in the previous case. 
We notice that
\begin{equation*}
	\lambda(\calS[-k, 0] \cap \calS_{\infty}) \,=\, \lambda(\calS[-k, 0] \cap \calS[1, \infty)) + \lambda(\calS[-k, 0] \cap (\calS_1 \setminus \calS[1, \infty))
\end{equation*}
and thus we can rewrite $Z_k$ as follows
\begin{align*}
	Z_k
	& \,=\, \langle\bbE[\lambda(\calS_1 \setminus \calS[1, \infty)) \mid \calG_1]\rangle - \bbE[\lambda((\calS_1 \setminus \calS[1, \infty]) \cap \calS[-k ,0]) \mid \calG_1] \\
	& \ \ \ \  + \bbE[\lambda(\calS[-k, 0] \cap \calS_{\infty}) \mid \calG_0] - \bbE[\lambda(\calS[-k, 0] \cap \calS[1, \infty)) \mid \calG_1].
\end{align*}
Before we let $k$ tend to infinity in the above expression, we rewrite the expression from the second line. We observe that
\begin{equation*}
	\lambda(\calS[-k, 0] \cap \calS_{\infty}) \,=\, \int_{\bbR^d} \bbjedan_{\calS[-k, 0]}(y)\, \bbjedan_{\calS_{\infty}}(y) \,\D y.
\end{equation*}
By taking conditional expectation with respect to $\calG_0$, we obtain
\begin{equation*}
	\bbE[\lambda(\calS[-k, 0] \cap \calS_{\infty}) \mid \calG_0] \,=\, \int_{\bbR^d} \bbjedan_{\calS[-k, 0]}(y)\, \bbE[\bbjedan_{\calS_{\infty}}(y)]\, \D y \,=\, \int_{\bbR^d} \bbjedan_{\calS[-k, 0]}(y)\, \phi(y)\, \D y ,
\end{equation*}
where we set 
$$\phi(y) = \bbP(y \in \calS_{\infty}).$$  
Similarly, we write
\begin{align*}
	\lambda(\calS[-k, 0] \cap \calS[1, \infty))
	& \,=\, \int_{\bbR^d} \bbjedan_{\calS[-k, 0] - X_1}(y)\, \bbjedan_{\calS[1, \infty) - X_1}(y)\, \D y
\end{align*}
and we take conditional expectation with respect to $\calG_1$ which yields
\begin{equation}
	\begin{aligned}\label{al:g(y-X_1)}
		\bbE[\lambda(\calS[-k, 0] \cap \calS[1, \infty)) \mid \calG_1]
		& \,=\, \int_{\bbR^d} \bbjedan_{\calS[-k, 0] - X_1}(y)\, \bbE[\bbjedan_{\calS[1, \infty) - X_1}(y)] \,\D y \\
		& \,=\, \int_{\bbR^d} \bbjedan_{\calS[-k, 0]}(y)\, \phi(y - X_1)\, \D y .
	\end{aligned}
\end{equation}
It follows that
\begin{equation}\label{eq:help_111}
\begin{split}
	&\bbE[\lambda(\calS[-k, 0]
	 \cap \calS_{\infty}) \mid \calG_0] - \bbE[\lambda(\calS[-k, 0] \cap \calS[1, \infty)) \mid \calG_1] \\
	& \,=\, \int_{\bbR^d} \bbjedan_{\calS[-k, 0]}(y)\, \bigl(\phi(y) - \phi(y - X_1)\bigr) \,\D y.
\end{split}
\end{equation}
We prove in \Cref{lm:L1} that the right-hand side integral in \cref{eq:help_111} 
is a well-defined random variable in ${\rm L}^1$. Thus, the  dominated convergence theorem implies \cref{eq:Z_k->Z}
 with
\begin{align*}
	Z
	& \,=\, \langle\bbE[\lambda(\calS_1 \setminus \calS[1, \infty)) \mid \calG_1]\rangle - \bbE[\lambda((\calS_1 \setminus \calS[1, \infty]) \cap \calS(-\infty ,0]) \mid \calG_1] \\
	& \ \ \ \  + \int_{\bbR^d} \bbjedan_{\calS(-\infty, 0]}(y)\, \bigl(\phi(y) - \phi(y - X_1)\bigr)\, \D y.
\end{align*}

\noindent
\textit{Step 2b}. 
We next show that $\bbE[|Z|]>0$ and we remark that the following arguments apply to all $d>3\alpha /2$.
From \cref{eq:key_relation,eq:pres} we have
\begin{equation*}\label{eq:H_n-1st}
	\langle \calV_n \rangle \,=\, \sum_{k = 0}^{n - 1} Z \circ \vartheta^k + H_n,
\end{equation*}
where
\begin{equation}\label{Def:H_n}
	H_n \,=\, \langle \bbE[\lambda(\calS_n \cap \calS[n, \infty)) \mid \calG_n] \rangle + \sum_{k = 0}^{n - 1}(Z_k - Z)\circ \vartheta^k.
\end{equation}
\Cref{eq:LM2.1A} yields
\begin{equation*}
	\langle \calV_n \rangle\, = \,\calV_n - \bbE[\calV_n] \le \calV_n - \bbE[\lambda(\calS_n \setminus \calS[n, \infty))]\, =\, \calV_n - n \Capa(\sB).
\end{equation*}
This implies
\begin{equation*}
	\calV_n \,\ge\, n \Capa(\sB) + \sum_{k = 0}^{n - 1} Z\circ \vartheta^k + H_n.
\end{equation*}
We aim to prove that there is $\bar c>0$ (which does not depend on $n$) such that for all $n\in \N$
\begin{equation}\label{eq:c_bar}
	\bbP\bigl(\{\calV_n \le \bar c\} \cap \{\aps{H_n} \le \bar c\}\bigr) \,>\, 0.
\end{equation}
Notice that if $\bbE[|Z|]=0$ then the inequality  
$\calV_n \ge n \Capa(\sB) + H_n$ 
and \cref{eq:c_bar} would imply that $\Capa(\sB) = 0$ which is a contradiction. Therefore, it is enough to show  \cref{eq:c_bar}.
From \cref{eq:def_of_Zk,eq:shape_of_Z} we obtain
\begin{equation}
\begin{aligned}\label{al:sum(Z-Zk)-1}
	\sum_{k = 0}^{n - 1} (Z - Z_k) \circ \vartheta^k
		& \,=\, \int_{\bbR^d} \bbjedan_{\calS(-\infty, 0]}(y)\, \bbE[\bbjedan_{\calS_{\infty}}(y) \mid \calG_0] \,\D y \\
		& \ \ \ \ - \sum_{k = 0}^{n - 2} \int_{\bbR^d} \Big( \bbjedan_{\calS(-\infty, 0] \setminus \calS_k}(y)\, \bbE[\bbjedan_{\calS[k, \infty)}(y) \mid \calG_{k + 1}] \\
		& \qquad\qquad \qquad \ \  - \bbjedan_{\calS(-\infty, 0] \setminus \calS_{k + 1}}(y)\, \bbE[\bbjedan_{\calS[k + 1, \infty)}(y) \mid \calG_{k + 1}] \Big)\, \D y \\
		& \ \ \ \ - \int_{\bbR^d} \bbjedan_{\calS(-\infty, 0] \setminus \calS_{n - 1}}(y)\, \bbE[\bbjedan_{\calS[n - 1, \infty)}(y) \mid \calG_n] \, \D y .
\end{aligned}
\end{equation}
We next observe that
\begin{equation}
	\begin{aligned}\label{al:sum(Z-Zk)-2}
		&\int_{\bbR^d} \Big(
		 \bbjedan_{\calS(-\infty, 0] \setminus \calS_k}(y)\, \bbE[\bbjedan_{\calS[k, \infty)}(y) \mid \calG_{k + 1}]\\
		 &\hspace{1cm} - \bbjedan_{\calS(-\infty, 0] \setminus \calS_{k + 1}}(y)\, \bbE[\bbjedan_{\calS[k + 1, \infty)}(y) \mid \calG_{k + 1}] \Big)\, \D y \\
		& \,=\, \int_{\bbR^d} \Big( \bbjedan_{\calS(-\infty, 0]}(y) \,\bbjedan_{\calS_k^c}(y)\, \bbE[\bbjedan_{\calS[k, k+1]}(y) + \bbjedan_{\calS[k + 1, \infty)}(y)\, \bbjedan_{\calS[k, k + 1]^c}(y) \mid \calG_{k + 1}] \\
		& \hspace{1.6cm} - \bbjedan_{\calS(-\infty, 0]}(y)\, \bbjedan_{\calS_{k + 1}^c}(y) \,\bbE[\bbjedan_{\calS[k + 1, \infty)}(y) \mid \calG_{k + 1}] \Big)\, \D y \\
		& \,=\, \int_{\bbR^d} \bbjedan_{\calS(-\infty, 0]}(y)\, \Big(\bbjedan_{\calS_k^c}(y)\, \bbjedan_{\calS[k, k + 1]}(y)\\
		& \hspace{3.6cm} + \bbjedan_{\calS_k^c}(y)\, \bbjedan_{\calS[k, k + 1]^c}(y)\, \bbE[\bbjedan_{\calS[k + 1, \infty)}(y) \mid \calG_{k + 1}] \\
		& \hspace{3.6cm} - \bbjedan_{\calS_{k + 1}^c}(y)\, \bbE[\bbjedan_{\calS[k + 1, \infty)}(y) \mid \calG_{k + 1}] \Big)\, \D y \\
		& \,=\, \lambda(\calS(-\infty, 0] \cap (\calS[k, k + 1] \setminus \calS_k)) .
	\end{aligned}
\end{equation}
We clearly have $\bbjedan_{\calS[n - 1, \infty)}(y) = \bbjedan_{\calS[n, \infty)}(y) + \bbjedan_{\calS[n - 1, n] \setminus \calS[n, \infty)}(y)$. This identity and a similar argument as in \cref{al:g(y-X_1)} yield
\begin{equation}
	\begin{aligned}\label{al:sum(Z-Zk)-3}
		&\int_{\bbR^d}
		 \bbjedan_{\calS(-\infty, 0] \setminus \calS_{n - 1}}(y)\, \bbE[\bbjedan_{\calS[n - 1, \infty)}(y) \mid \calG_n] \,\D y \\
		& \,=\, \int_{\bbR^d} \bbjedan_{\calS(-\infty, 0] \setminus \calS_{n - 1}}(y)\, \bbE[\bbjedan_{\calS[n, \infty)}(y) \mid \calG_n]\, \D y\\
		& \ \ \ \ + \int_{\bbR^d} \bbjedan_{\calS(-\infty, 0] \setminus \calS_{n - 1}}(y)\, \bbE[\bbjedan_{\calS[n - 1, n] \setminus \calS[n, \infty)}(y) \mid \calG_n]\, \D y \\
		&\, =\, \int_{\bbR^d} \bbjedan_{\calS(-\infty, 0]}(y)\, \phi(y - X_n)\, \D y - \int_{\bbR^d} \bbjedan_{\calS(-\infty, 0] \cap \calS_{n-1}}(y)\, \phi(y - X_n)\, \D y \\
		& \ \ \ \ + \int_{\bbR^d} \bbjedan_{\calS(-\infty, 0] \setminus \calS_{n - 1}}(y)\, \bbE[\bbjedan_{\calS[n - 1, n] \setminus \calS[n, \infty)}(y) \mid \calG_n]\, \D y.
	\end{aligned}
\end{equation}
By combining \cref{al:sum(Z-Zk)-1,al:sum(Z-Zk)-2,al:sum(Z-Zk)-3}, we arrive at
\begin{equation}
	\begin{aligned}\label{al:sum(Z-Zk)}
		&\sum_{k = 0}^{n - 1} (Z - Z_k) \circ \vartheta^k\\
		& \,=\, \int_{\bbR^d} \bbjedan_{\calS(-\infty, 0]}(y)\, \bigl(\phi(y) - \phi(y - X_n)\bigr)\, \D y  + \int_{\bbR^d} \bbjedan_{\calS(-\infty, 0] \cap \calS_{n-1}}(y)\, \phi(y - X_n)\,\D y \\
		& \ \ \ \  - \int_{\bbR^d} \bbjedan_{\calS(-\infty, 0] \setminus \calS_{n - 1}}(y)\, \bbE[\bbjedan_{\calS[n - 1, n] \setminus \calS[n, \infty)}(y) \mid \calG_n]\, \D y  - \lambda(\calS(-\infty, 0] \cap \calS_{n - 1}).
	\end{aligned}
\end{equation}
We claim that there is a constant $ \tilde c > 0$ such that for all $n\in \N$
\begin{equation}\label{toShow_222}
	\bbP \OBL{\VIT{\sup_{0\le s \le n} \aps{X_s} \le 1} \cap \VIT{ \APS{\int_{\bbR^d}\!\!\!\! \bbjedan_{\calS(-\infty, 0]}(y)\, \bigl(\phi(y) - \phi(y - X_n)\bigr) \,\D y\,} \le \tilde c+1}} > 0.
\end{equation}
If $\sup_{0\le s\le n}|X_n|\le 1$, then clearly $\lambda(\calS_n)\le \lambda(\sB(0,2))$ and this allows us to estimate the first term on the right-hand side of \cref{Def:H_n} and, similarly, the three last terms on the right-hand side of \cref{al:sum(Z-Zk)} by a constant. 
We infer that  there exists a constant $\bar c > 0$ such that 
\begin{align*}
	\VIT{\sup_{0\le s \le n} \aps{X_s} \le 1} &
	 \cap \VIT{ \APS{\int_{\bbR^d} \bbjedan_{\calS(-\infty, 0]}(y)\, \bigl(\phi(y) - \phi(y - X_n)\bigr)\,\D y\,} \le \tilde c +1}\\
	 &\, \subseteq\, \{\calV_n \le \bar c\} \cap \{H_n \le \bar c\}.
\end{align*}
To finish the proof of \cref{eq:c_bar},
we are only left to show \cref{toShow_222}. In view of the Markov inequality it is enough to prove that under $\sup_{0\le s\le n}|X_n|\le 1$ we have
\begin{align*}
\bbE \Big[  \Big|\int_{\bbR^d} \bbjedan_{\calS(-\infty, 0]}(y)\, \bigl(\phi(y) - \phi(y - X_n)\bigr) \,\D y\,\Big|  \Big] \,<\,\infty.
\end{align*}
This holds as, under $\sup_{0\le s\le n}|X_n|\le 1$,
\begin{align*}
\bbE \Big[  \Big|\int_{\bbR^d} \bbjedan_{\calS(-\infty, 0]}(y)\, \bigl(\phi(y) - \phi(y - X_n)\bigr) \,\D y\,\Big|  \Big] \le 
	\sup_{x\in \sB}\int_{\bbR^d}\phi(y)|\phi(y)-\phi(y-x)|\,\D y\,<\,\infty,
\end{align*}
where convergence of the last integral is established in  \Cref{lm:bound}.
\medskip

\noindent
\textit{Step 2c}. 	We finally show that the limit in \cref{eq:lim_Var=lim_Y} is positive.  \Cref{eq:Z_k->Z}
implies
	\begin{equation*}
		\lim_{n \nearrow \infty}\frac{1}{n} \sum_{k = 0}^{n - 1} \bbE[\aps{Z_k}]\, =\, \bbE[\aps{Z}].
	\end{equation*}
Hence, by Jensen's inequality, 
	\begin{align*}
		\lim_{n \nearrow \infty}\frac{1}{n}\sum_{k = 0}^{n - 1}\bbE[Z_k^2]
		& \,\ge\, \lim_{n \nearrow \infty}
		\bbE\Big[\Big(\frac{1}{n}\sum_{k = 0}^{n - 1}\aps{Z_k}\Big)^2\Big] \,\ge\,\lim_{n \nearrow \infty} \Big(\bbE \Big[ \frac{1}{n} \sum_{k = 0}^{n - 1} \aps{Z_k}\Big]\Big)^2 \\
		&\,=\, \lim_{n \nearrow \infty}\Big(\frac{1}{n} \sum_{k = 0}^{n - 1} \bbE[\aps{Z_k}]\Big)^2\,=\, (\bbE[\aps{Z}])^2.
	\end{align*}
By \cref{eq:pres}, we have
	\begin{equation*}
		\lim_{n \nearrow \infty} \frac{1}{n} \sum_{k = 1}^n \bbE[Y_k^2] \,=\, \lim_{n \nearrow \infty} \frac{1}{n} \sum_{k = 0}^{n - 1} \bbE[Z_k^2] \,\ge\, (\bbE[\aps{Z}])^2 \,>\, 0,
	\end{equation*}
and this finishes the proof of the lemma.
\end{proof}

\subsection{Proof of the central limit theorem}

In this paragraph, we prove \Cref{tm:CLT}. In the proof we apply the Lindeberg-Feller central limit theorem which we include for completeness. 

\begin{lemma}[{\cite[Theorem 3.4.5]{Durrett}}]\label{lm:Durrett}
	For each $n \in\N$ let $\{X_{n, i}\}_{1 \le i \le n}$ be a sequence of independent random variables with zero mean. If	the following conditions are satisfied
		\begin{itemize}
			\item [(i)] $\displaystyle\lim_{n\nearrow\infty}\sum_{i = 1}^n \bbE[X_{n,i}^2] = \sigma^2 > 0$,	and
			\item [(ii)] for every $\varepsilon > 0$, $\displaystyle\lim_{n\nearrow\infty}\sum_{i = 1}^n \bbE\left[X_{n,i}^2 \bbjedan_{\{|X_{n,i}| > \varepsilon\}} \right] =0,$
	\end{itemize}
	then $X_{n,1} + \cdots +  X_{n,n} \xrightarrow[n\nearrow\infty]{\text{{\rm (d)}}} \sigma\,\calN(0,1).$
\end{lemma}

\begin{proof}[Proof of \Cref{tm:CLT}]
	For $t>0$ large enough we choose $n = n(t) = \floor{\log(t)}$. We have
	\begin{equation*}
		\calV_t \,=\, \lambda(\calS_t) \,=\, \lambda\bigl(\calS_{t/n} \cup \calS[t/n, t]\bigr) \,=\, \lambda\bigl((\calS_{t/n} - X_{t/n}) \cup (\calS[t/n, t] - X_{t/n})\bigr).
	\end{equation*}
	By the Markov property,
	\begin{equation*}
		\calS^{(1)}_{t/n} \,= \, \calS_{t/n} - X_{t/n} \qquad \textnormal{and} \qquad \calS^{(2)}_{(n - 1)t/n} \,=\, \calS[t/n, t] - X_{t/n}
	\end{equation*}
	are independent, and $\calS^{(2)}_{(n - 1)t/n}$ has the same law as $\calS_{(n - 1)t/n}$. Rotational invariance of  $\X$ implies that $\calS^{(1)}_{t/n}$ is equal in law to $\calS_{t/n}$. Hence,
	\begin{equation*}
		\calV_t \,=\, \lambda(\calS^{(1)}_{t/n}) + \lambda(\calS^{(2)}_{(n - 1)t/n}) - \lambda\bigl(\calS^{(1)}_{t/n} \cap \calS^{(2)}_{(n - 1)t/n}\bigr).
	\end{equation*}
	By iterating this procedure, we obtain 
	\begin{equation}\label{eq:decom_of_Vt}
		\calV_t \,=\, \sum_{i = 1}^n \lambda(\calS^{(i)}_{t/n}) - \sum_{i = 1}^{n - 1}\lambda\bigl(\calS^{(i)}_{t/n} \cap \calS^{(i + 1)}_{(n - i)t/n}\bigr).
	\end{equation}
We denote
	\begin{equation*}
		\calV^{(i)}_{t/n} \,=\, \lambda(\calS^{(i)}_{t/n}) \qquad \textnormal{and} \qquad \calR(t) \,=\, \sum_{i = 1}^{n - 1}\lambda\bigl(\calS^{(i)}_{t/n} \cap \calS^{(i + 1)}_{(n - i)t/n}\bigr),
	\end{equation*}
	and we notice that $\{\calV^{(i)}_{t/n}\}_{1 \le t \le n}$ are i.i.d.\ random variables. 
	By taking expectation in \cref{eq:decom_of_Vt} and then subtracting, we obtain
	\begin{equation}\label{eq:rel_for_<Vt>}
		\langle \calV_t \rangle \,=\, \sum_{i = 1}^n \langle \calV^{(i)}_{t/n} \rangle - \langle \calR(t) \rangle.
	\end{equation}
	We first show that
	\begin{equation}\label{eq:error_term->0}
		\frac{\langle \calR(t) \rangle}{\sqrt{t}}\, \xrightarrow[t\nearrow\infty]{{\rm L}^1}\, 0.
	\end{equation}
	Since $\calR(t) \ge 0$, we clearly have $\bbE[\aps{\langle \calR(t) \rangle}] \le 2\,\bbE[\calR(t)]$. By \Cref{Cor:moment_bounds},
	\begin{equation*}
		\bbE[\calR(t)] \, \le\, \sum_{i = 1}^{n - 1} \bbE\left[\lambda\bigl(\calS^{(i)}_{t/n} \cap \calS^{(i + 1)}_{\infty}\bigr)\right]
		\, \le\, c\, n\, h(t/n),
	\end{equation*}
	for all $t>0$ large enough. Hence, \cref{eq:error_term->0} follows by
\cref{eq:def_of_hd}, and the fact that $n = \floor{\log(t)}$ and $d/\alpha>3 /2$.
Next we prove that
\begin{equation}\label{eq:sum->N}
	\frac{1}{\sqrt{t}} \sum_{i = 1}^n \langle \calV^{(i)}_{t/n} \rangle \,\xrightarrow[t\nearrow\infty]{\text{(d)}}\, \sigma\,\calN(0,1).
\end{equation}
For this we introduce the random variables
\begin{equation*}
	X_{n, i} \,=\, \frac{\langle \calV^{(i)}_{t/n} \rangle}{\sqrt{t}},\qquad i=1,\dots,n,
\end{equation*}
and we check the validity of conditions (i) and (ii) from \Cref{lm:Durrett}. 
Condition (i) follows by \Cref{lm:positivity}, 
\begin{equation}\label{eq:Durret-(i)}
	\lim_{n \nearrow \infty}\sum_{i = 1}^n \bbE[X_{n,i}^2] 
	\,=\, \lim_{n \nearrow \infty}\frac{n}{t} \Var(\calV_{t/n}) 
	\, = \, \sigma^2.
\end{equation} 
To establish condition (ii) we apply the Cauchy-Schwartz inequality and obtain that for every $\varepsilon > 0$,
\begin{equation*}
	\bbE\left[X_{n,i}^2 \,\bbjedan_{\{|X_{n,i}| > \varepsilon\}} \right]
\, \le \, \frac{1}{t} \OBL{\bbE\UGL{\langle \calV_{t/n} \rangle^4} \bbP\OBL{\aps{\langle \calV_{t/n} \rangle} > \varepsilon \sqrt{t}}}^{1/2}.
\end{equation*}
By Chebyshev's inequality combined with \Cref{lm:positivity} and the fact that $n = \floor{\log(t)}$, there is a constant $c_1>0$ such that
\begin{equation*}
	\bbP\OBL{\aps{\langle \calV_{t/n} \rangle} > \varepsilon \sqrt{t}}
\,\le\, \frac{\Var(\calV_{t/n})}{\varepsilon^2 t}	
	\, \le\, \frac{c_1 t/n}{\varepsilon^2 t} \,=\, \frac{c_1}{\varepsilon^2 n}.
\end{equation*} 
This together with \Cref{lm:4-th-moment} imply that there are constants $c_2, c_3>0$ such that
\begin{equation}\label{eq:Durret-(ii)}
	\lim_{n \nearrow \infty}\sum_{i = 1}^n \bbE\left[X_{n,i}^2 \,\bbjedan_{\{|X_{n,i}| > \varepsilon\}} \right] 
	\,\le\, \lim_{n \nearrow \infty}\sum_{i = 1}^n \frac{1}{t} \Big(c_2 \Big(\frac{t}{n}\Big)^2 \frac{c_1}{\varepsilon^2n}\Big)^{1/2} 
	\,\le\, \lim_{n \nearrow \infty}\frac{c_3}{\sqrt{n}} \, =\,0.
\end{equation} 
Thus, \cref{eq:sum->N} follows and we conclude that
\begin{equation*}
	\frac{\langle \calV_t \rangle}{\sqrt{t}}\, \xrightarrow[t\nearrow\infty]{\text{(d)}}\, \sigma\,\calN(0,1).
\end{equation*}
We finally observe that 
\begin{align*}
	\frac{\calV_t - t\Capa(\sB)}{\sigma \sqrt{t}} 
	\,=\, \frac{\langle \calV_t \rangle}{\sigma \sqrt{t}} + \frac{\bbE[\calV_t] - t\Capa(\sB)}{\sigma \sqrt{t}}, 
\end{align*}
which allows us to finish the proof in view of \Cref{lm:obtaining_tC(B)} and \Cref{Cor:moment_bounds}.
\end{proof}

\section{Functional central limit theorem}\label{Sec:FCLT}

The goal of this section is to prove the functional central  limit theorem in \cref{Result:FCLT}. To prove this statement we adapt the proof of \cite[Theorem 1.1]{CSS19_2},
 which is concerned with the functional central limit theorem for the capacity of the range of a stable random walk.

We again assume  that $\X$ is a stable rotationally invariant L\'{e}vy process in $\bbR^d$ of index $\alpha \in (0,2]$ satisfying $d/\alpha> 3/2$.
We follow the classical two-step scheme (see \cite[Theorem 16.10 and Theorem 16.11]{Kallenberg}). Let $\{Y^{n}\}_{n \ge 1}$ be a sequence of random elements in the Skorohod space $\calD([0, \infty), \bbR)$ endowed with the Skorohod $\text{J}_1$ topology. The sequence $\{Y^{n}\}_{n \ge 1}$ converges weakly to a random element $Y$ (in $\calD([0, \infty), \bbR)$) if the following two conditions are satisfied: 
\begin{enumerate}[(i)]
	\item The finite dimensional distributions of $\{Y^{n}\}_{n \ge 1}$ converge weakly to the finite dimensional distributions of $Y$.
	\item For any bounded sequence $\{T_n\}_{n \ge1}$ of $\{Y^{n}\}_{n \ge 1}$-stopping times and any non-negative sequence $\{b_n\}_{n \ge 1}$  converging to zero,
		\begin{equation*}
		\lim_{n \nearrow \infty}\bbP \bigl(|Y_{T_n + b_n}^{n} - Y_{T_n}^{n}| \ge \varepsilon\bigr)\, =\, 0,\qquad  \varepsilon >0.
	\end{equation*}
\end{enumerate}

\begin{theorem}\label{tm:FCLT} 
Under the above assumptions, the following convergence holds
	\begin{equation*}
		\VIT{\frac{\calV_{nt} - nt\Capa(\sB)}{\sigma \sqrt{n}}}_{t \ge 0}\, \xrightarrow[n\nearrow\infty]{({\rm J}_1)}\, \{W_t\}_{t \ge 0},
	\end{equation*}
	where $\sigma$ is the constant from  \Cref{tm:CLT}.
\end{theorem}
\begin{proof}
We consider the following sequence of random elements which are defined in the space $\calD([0, \infty), \bbR)$,
\begin{equation}\label{eq:def_of_Y^n_t}
	Y_t^{n} \,=\,\frac{\calV_{nt} - nt\Capa(\sB)}{\sigma \sqrt{n}},\qquad  n\in\N, 
\end{equation}
where $\sigma$ is the constant from \Cref{tm:CLT}. Let us start by showing condition (i).
\medskip

\noindent
\textit{Condition (i)}.
By \Cref{tm:CLT}, we have
\begin{equation*}
	Y^n_t \,=\, \frac{\calV_{nt} - nt\Capa(\sB)}{\sigma \sqrt{n}} \,=\,\sqrt{t}\, \frac{\calV_{nt} - nt\Capa(\sB)}{\sigma \sqrt{nt}} \, \xrightarrow[n\nearrow\infty]{(\text{d})}\, \calN(0, t).
\end{equation*}
Let $k \in\N$ be arbitrary  and choose $0 = t_0 < t_1 < t_2 < \cdots < t_k $. We need to prove that
\begin{equation*}
	(Y_{t_1}^{n}, Y_{t_2}^{n}, \ldots, Y_{t_k}^{n})\, \xrightarrow[n\nearrow\infty]{(\text{d})}\, (W_{t_1}, W_{t_2}, \ldots, W_{t_k}).
\end{equation*}
 In view of the Cram\'{e}r-Wold theorem \cite[Corollary 5.5]{Kallenberg} it suffices to show that
\begin{equation*}\label{eq:cvg_CW}
	\sum_{j = 1}^k \xi_j Y_{t_j}^{n}\, \xrightarrow[n\nearrow\infty]{(\text{d})}\, \sum_{j = 1}^k \xi_j W_{t_j}, \qquad  (\xi_1, \xi_2, \ldots, \xi_k) \in \bbR^k.
\end{equation*}
Using a similar reasoning as in the beginning of the proof of \Cref{tm:CLT}, we obtain for $j \in \{1, 2, \ldots, k\}$,
\begin{equation*}
	\calV_{nt_j}\, =\, \sum_{i = 1}^j \calV^{(i)}_{n(t_i - t_{i - 1})} - \sum_{i = 1}^{j - 1} \calR^{(i)}_{nt_j},
\end{equation*}
where 
\begin{equation*}
	\calV^{(i)}_{n(t_i - t_{i - 1})} \,=\, \lambda(\calS^{(i)}_{n(t_i - t_{i - 1})}) \qquad \textnormal{and} \qquad \calR^{(i)}_{nt_j} \,=\, \lambda\bigl(\calS^{(i)}_{n(t_i - t_{i - 1})} \cap \calS^{(i + 1)}_{n(t_j - t_i)}\bigr).
\end{equation*}
The random variables $\calV^{(i)}_{n(t_i - t_{i - 1})}$, for $i \in \{1, 2, \ldots, k\}$, are independent, $\calS^{(i)}_{n(t_i - t_{i - 1})}$ has the same law as $\calS_{n(t_i - t_{i - 1})}$, and $\calR^{(i)}_{nt_j}$ has the same law as $\lambda\bigl(\calS_{n(t_i - t_{i - 1})} \cap \calS'_{n(t_j - t_i)}\bigr)$, with $\calS'_{n(t_j - t_i)}$ being an independent copy of $\calS_{n(t_j - t_i)}$. 
For arbitrary $(\xi_1, \xi_2, \ldots, \xi_k) \in \bbR^k$ we have
\begin{align*}
		\sum_{j = 1}^k \xi_j Y^n_{t_j}
		& \,=\, \sum_{j = 1}^k \xi_j \OBL{\frac{\calV_{nt_j} - nt_j \Capa(\sB)}{\sigma \sqrt{n}}} \\
		& \,=\, \frac{1}{\sigma \sqrt{n}}\sum_{j = 1}^k \xi_j \OBL{\sum_{i = 1}^j \calV^{(i)}_{n(t_i - t_{i - 1})} - \sum_{i = 1}^{j - 1} \calR^{(i)}_{nt_j} - \sum_{i = 1}^j n(t_i - t_{i - 1}) \Capa(\sB)} \\
		& \,=\, \sum_{j = 1}^k \xi_j \sum_{i = 1}^j \frac{\calV^{(i)}_{n(t_i - t_{i - 1})} - n(t_i - t_{i - 1}) \Capa(\sB)}{\sigma \sqrt{n}} - \frac{1}{\sigma} \sum_{j = 1}^k \xi_j \sum_{i = 1}^{j - 1} \frac{\calR^{(i)}_{nt_j}}{\sqrt{n}} \\
		& \,=\, \sum_{i = 1}^k \OBL{\sum_{j = i}^k \xi_j} \frac{\calV^{(i)}_{n(t_i - t_{i - 1})} - n(t_i - t_{i - 1}) \Capa(\sB)}{\sigma \sqrt{n}} - \frac{1}{\sigma} \sum_{j = 1}^k \xi_j \sum_{i = 1}^{j - 1} \frac{\calR^{(i)}_{nt_j}}{\sqrt{n}}.
\end{align*}
\Cref{tm:CLT} provides that
\begin{equation*}
	\frac{\calV^{(i)}_{n(t_i - t_{i - 1})} - n(t_i - t_{i - 1}) \Capa(\sB)}{\sigma \sqrt{n}} \,\xrightarrow[n\nearrow\infty]{(\text{d})}\, \calN(0, t_i - t_{i - 1}).
\end{equation*}
Markov's inequality combined with \Cref{Cor:moment_bounds} implies that for every $\varepsilon > 0$,
\begin{align*}
\bbP\OBL{\frac{\calR^{(i)}_{nt_j}}{\sqrt{n}} > \varepsilon}
	& \,\le\, \frac{\bbE[\lambda(\calS_{n(t_i - t_{i - 1})} \cap \calS'_{n(t_j - t_i)})]}{\varepsilon \sqrt{n}} \\
	& \,\le\, \frac{\bbE[\lambda(\calS_{nt_j} \cap \calS'_{\infty})]}{\varepsilon \sqrt{n}} 
	 \,\le\,  \frac{c\, h(nt_j)}{\varepsilon \sqrt{n}},
\end{align*}
which converges to zero, as $n$ tends to infinity.
Since for $i \in \{1, 2, \ldots, k\}$ the random variables $\calV^{(i)}_{n(t_i - t_{i - 1})}$ are independent, we obtain
\begin{align*}
\sum_{j = 1}^k \xi_j Y^n_{t_j} \,\xrightarrow[n\nearrow\infty]{(\text{d})}\, 
\calN\left(0, \sum_{i = 1}^k \Bigl(\sum_{j = i}^k \xi_j\Bigr)^2 (t_i - t_{i - 1})\right).
\end{align*}
It follows that the finite dimensional distributions of $\{Y^{n}\}_{n\ge1}$ converge weakly to the finite dimensional distributions of a one-dimensional  standard Brownian motion.
\medskip

\noindent
\textit{Condition (ii)}.
Let $\{T_n\}_{n \ge1}$ be a bounded sequence of $\{Y^{n}\}_{n\ge1}$-stopping times, and let $\{b_n\}_{n \ge 1} \subset [0,\infty)$ be an arbitrary sequence which converges to zero. We aim to prove that
\begin{equation*}\label{ToShow}
	Y^{n}_{T_n + b_n} - Y^{n}_{T_n}\, \xrightarrow[n\nearrow\infty]{\bbP}\, 0,
\end{equation*}
where the convergence holds in probability. By \cref{eq:def_of_Y^n_t}, we have
\begin{equation}\label{eq:Y-Y-1}
	Y^{n}_{T_n + b_n} - Y^{n}_{T_n} \,=\, \frac{\calV_{n(T_n + b_n)} - n(T_n + b_n)\Capa(\sB)}{\sigma \sqrt{n}} - \frac{\calV_{nT_n} - nT_n \Capa(\sB)}{\sigma \sqrt{n}}.
\end{equation}
The Markov property and rotational invariance of $\X$ yield
\begin{align*}
	\calV_{n(T_n + b_n)} - \calV_{nT_n}
	& \,=\, \lambda\bigl((\calS_{nT_n} \cup \calS[nT_n, n(T_n + b_n)]) - X_{nT_n}\bigr) - \lambda\bigl(\calS_{nT_n} - X_{nT_n}\bigr) \\
	& \,=\, \lambda(\calS^{(1)}_{nT_n}) + \lambda(\calS^{(2)}_{nb_n}) - \lambda\bigl(\calS^{(1)}_{nT_n} \cap \calS^{(2)}_{nb_n}\bigr) - \lambda(\calS^{(1)}_{nT_n})\\
	& \,=\, \lambda(\calS^{(2)}_{nb_n}) - \lambda\bigl(\calS^{(1)}_{nT_n} \cap \calS^{(2)}_{nb_n}\bigr),
\end{align*}
where $\calS^{(1)}_{nT_n}$ and $\calS^{(2)}_{nb_n}$ are independent and have the same distribution as $\calS_{nT_n}$ and $\calS_{nb_n}$, respectively. \Cref{eq:Y-Y-1} implies
\begin{align*}
	Y^{n}_{T_n + b_n} - Y^{n}_{T_n} \,=\, \frac{\lambda(\calS^{(2)}_{nb_n}) - \lambda\bigl(\calS^{(1)}_{nT_n} \cap \calS^{(2)}_{nb_n}\bigr) - nb_n \Capa(\sB)}{\sigma\sqrt{n}}.
\end{align*}
With a slight abuse of notation we write $\calV_{nb_n} = \lambda(\calS^{(2)}_{nb_n})$. By \Cref{lm:obtaining_tC(B)}, we obtain
\begin{align}\label{al:Y-Y-2}
		Y^{n}_{T_n + b_n} - Y^{n}_{T_n}
		& \,=\, \frac{\calV_{nb_b} - \bbE[\calV_{nb_n}]}{\sigma \sqrt{n}} \,+\, \frac{\bbE[\lambda\bigl(\calS_{nb_n} \cap \calS[nb_b, \infty)\bigr)]}{\sigma \sqrt{n}} \,-\, \frac{\lambda\bigl(\calS^{(1)}_{nT_n} \cap \calS^{(2)}_{nb_n}\bigr)}{\sigma \sqrt{n}}.
\end{align}
We prove that the three terms on the right-hand side of \cref{al:Y-Y-2} converge to zero in probability. For the first term, Chebyshev's inequality
yields that for every $\varepsilon > 0$,
\begin{equation*}
\bbP\OBL{\APS{\frac{\calV_{nb_n} - \bbE[\calV_{nb_n}]}{\sqrt{n}}} > \varepsilon} \,\le\, \frac{\Var(\calV_{nb_n})}{\varepsilon^2 n},
\end{equation*}
and we are left to show that
\begin{equation}\label{eq:term1-key}
\lim_{n \nearrow \infty} \frac{\Var(\calV_{nb_n})}{n} \,=\, 0.
\end{equation}
This follows by \Cref{lm:positivity}. Indeed, there exist  $t_1, c_1 > 0$, such that for every $t \ge t_1$, we have $
\Var(\calV_t) \,\le\,  c_1t$, and whence 
for $nb_n \ge t_1$, $\Var(\calV_{nb_n})\, \le\, c_1nb_n$.
For $nb_n < t_1$ we observe that
$\Var(\calV_{nb_n}) \,\le\, \bbE[\calV_{nb_n}^2]\, \le\,  \bbE[\calV_{t_1}^2]$.  This trivially implies \cref{eq:term1-key}. 

By \Cref{Cor:moment_bounds}, similarly as above, we show that there is $t_2>0$ such that for all $n\in \N$
\begin{equation*}
\bbE[\lambda\bigl(\calS_{nb_n} \cap \calS[nb_n, \infty)\bigr)] 
\,\le\, c\,h(nb_n) + \bbE[\calV_{t_2}].
\end{equation*}
We then easily conclude that the second term on the right-hand side of \cref{al:Y-Y-2} converges to zero in probability.

There exists  $c_2 > 0$ such that $\sup_{n \ge 1} T_n \le c_2$. By the  Markov inequality and \Cref{Cor:moment_bounds}, we obtain that for every $\varepsilon > 0$
\begin{align*}
\bbP\OBL{\frac{\lambda\bigl(\calS^{(1)}_{nT_n} \cap \calS^{(2)}_{nb_n}\bigr)}{\sigma \sqrt{n}} > \varepsilon}
		 \,\le\, \frac{\bbE[\lambda(\calS^{(1)}_{nT_n} \cap \calS^{(2)}_{nb_n})]}{\varepsilon\sigma \sqrt{n}} 
		\,\le\, \frac{\bbE[\lambda(\calS^{(1)}_{c_2n} \cap \calS^{(2)}_{\infty})]}{\varepsilon\sigma \sqrt{n}}  \,\le\, \frac{c\, h( c_2n)}{\varepsilon \sigma \sqrt{n}},
\end{align*}
which converges to zero, as $n$ tends to infinity.
This shows that the last term on the right-hand side of \cref{al:Y-Y-2} goes to zero in probability and the proof is finished.
\end{proof}

\section{Laws of the iterated logarithm}\label{Sec:asIP}
This section is devoted to the proof of the following result.

\begin{theorem}\label{tm:LIL} 
If $d/\alpha>9/5$, then  the process $\{\calV_t\}_{t\ge0}$ satisfies   Khintchine's and Chung's law of the iterated logarithm, that is, \cref{Result:LIL(Khi),Result:LIL(Chung)}, respectively.
	\end{theorem}

We start with the proof of  \cref{Result:LIL(Khi)}, which  is based on the following lemma.

\begin{lemma}[{\cite[Chapter X, Theorem 2]{Petrov}}]\label{lm:petrov} 
Let $\{Y_n\}_{n\ge1}$ be a sequence of independent random variables with mean $0$ and finite variance. Set 
$$S_n\,=\,\sum_{i=1}^nY_i,\qquad \text{and}\qquad  s^2_n\,=\,\sum_{i=1}^n\mathbb{E}[Y_i^2].$$
 Suppose  $\lim_{n\nearrow\infty}s_n=\infty,$ and that for any $\varepsilon>0$, 
	\begin{equation}\label{eq:11}
	\lim_{n \nearrow \infty}\frac{1}{s^2_n}\sum_{i=i_0}^n\mathbb{E}\left[Y_i^2\Ind_{\{|Y_i|\ge\varepsilon\sqrt{s_i^2/\log\log s_i^2}\}}\right]\,=\,0
	\end{equation} 
		and 
	\begin{equation}\label{eq:12}
	\sum_{i=i_0}^\infty\frac{1}{s_i^2\log\log s_i^2}\mathbb{E}\left[Y_i^2\Ind_{\{|Y_i|\ge\varepsilon\sqrt{s_i^2/\log\log s_i^2}\}}\right]\,<\,\infty,
	\end{equation} 
	where $i_0=\min\{i\ge1:\log\log s_j^2>0,\  j\ge i\}.$ 
	Then 
	$$\limsup_{n\nearrow \infty}\frac{S_n}{\sqrt{2s_n^2\log\log s_n^2}}\,=\,1\qquad \Prob\text{-a.s.}$$
\end{lemma}

\begin{proof}[Proof of \Cref{tm:LIL} - Khintchine's law of the iterated logarithm] 
By $\lfloor a\rfloor$ we denote the greatest integer less than or equal to  $a\in\R$. 
We consider a sequence $\{n_i\}_{i\ge0}$ of non-negative
integers such that if $2^k\leq n_i<2^{k+1}$, then $n_i$ runs over 
all consecutive  integers
of the form 
$2^k+\lfloor j2^{k/2}/k\rfloor$, for $k\in\N$ and  $j=0,1,\dots,\lfloor k 2^{k/2}\rfloor$. 
We set $n_0=0$.
 Clearly, if
  $2^k\le n_i< 2^{k+1}$
then
   $0\le n_{i+1}-n_i\le 2^{k/2}/k+1$. 
   Hence, 
   \begin{equation}\label{eq:2}
   \lim_{i\nearrow\infty}\frac{n_{i+1}}{n_i}\,=\,1\qquad \text{and}\qquad n_{i+1}-n_i\,=\,\mathcal{O}(n_i^{1/2}/\log n_i).
   \end{equation}
Since
	$$\sup_{n_i\le t\le n_{i+1}}|\langle \calV_t \rangle -\langle \calV_{n_i} \rangle|\,\le\, \calV[n_i,n_{i+1}]+\mathbb{E}[\calV[n_i,n_{i+1}]],$$
we see that in the case when $d/\alpha>1$ (transience)
 \cref{eq:SLLN,eq:SLLN2,eq:2} imply that
	\begin{equation}\label{eq:13}\lim_{i\nearrow\infty}\frac{\sup_{n_i\le t\le n_{i+1}}|\langle \calV_t \rangle -\langle \calV_{n_i} \rangle|}{\sqrt{n_i/\log\log n_i}}\,=\,0\qquad \Prob\text{-a.s.}\end{equation}
Thus, it suffices to prove that $\Prob$-a.s.
	$$\liminf_{i\nearrow \infty}\frac{\langle\calV_{n_i}\rangle}{\sqrt{2\sigma^2n_i\log\log n_i}}\,=\,-1\qquad \text{and}\qquad \limsup_{i\nearrow \infty}\frac{\langle\calV_{n_i}\rangle}{\sqrt{2\sigma^2n_i\log\log n_i}}\,=\,1.$$ 
	We only discuss the second relation as the first one can be handled in an analogous way. 
	For $i\ge0$ we set 
	\begin{align}\label{J_process}
	\mathcal{V}(n_i,n_{i+1}]\,=\,\lambda(\calS(n_i,n_{i+1}])\qquad \text{and}\qquad \mathcal{J}_{n_i}\,=\,\lambda(\calS(n_{i},n_{i+1}]\cap\calS_{n_i}),
\end{align}	  
	where $\calS(s,t]=\bigcup_{s<u\le t}\{X_u+\sB\}$. 
	Observe that $\{\mathcal{V}(n_i,n_{i+1}]\}_{i\ge0}$ forms a sequence of independent random variables, and for $i\geq 1$ we have
	\begin{equation}\label{eq:dec}\calV_{n_i}\,=\,\sum_{j=0}^{i-1}\mathcal{V}(n_j,n_{j+1}]-\sum_{j=0}^{i-1}\mathcal{J}_{n_j},
	\end{equation}
	which yields
	$$\frac{\langle\calV_{n_i}\rangle}{\sqrt{2\sigma^2n_i\log\log n_i}}\,=\,\frac{\sum_{j=0}^{i-1}\langle\mathcal{V}(n_j,n_{j+1}]\rangle}{\sqrt{2\sigma^2n_i\log\log n_i}}-\frac{\sum_{j=0}^{i-1}\langle\mathcal{J}_{n_j}\rangle}{\sqrt{2\sigma^2n_i\log\log n_i}}.$$
In \Cref{lm:4.3'} we show that if
 $d/\alpha>9/5$, then
	the last term on the right-hand side of the above identity converges to zero $\Prob$-a.s.
	Thus we are left to prove that
	$$\limsup_{i\nearrow \infty}\frac{\sum_{j=0}^{i-1}\langle\mathcal{V}(n_j,n_{j+1}]\rangle}{\sqrt{2\sigma^2n_i\log\log n_i}}\,=\,1\qquad \Prob\text{-a.s.}$$
	 When $d/\alpha\in(1,2)$ we  set $\Lambda=2-d/\alpha$ which satisfies $\Lambda \in (0,1)$.
	We apply \Cref{lm:var-asimp} and obtain that for $n_i\in[2^{k},2^{k+1}]$ there are constants $c_1,c_2,c_3,c_4>0$ such that 
	\begin{align*}
	\Var\left(\sum_{j=0}^{i-1}\mathcal{V}(n_j,n_{j+1}]\right)
	&\,\le\,\sigma^2n_i+c_1\left(\sum_{j=0}^{i-1}(n_{j+1}-n_j)^{1/2}h(n_{j+1}-n_j)\right)\\
	&\,\le\, \sigma^2n_i+c_2\left(\sum_{l=1}^{k}2^{3l/4}h(2^{l/2}/l)l^{1/2}\right)\\
	&\,\le\,\begin{cases}
	\sigma^2n_i+c_3\left(\sum_{l=1}^{k}2^{3l/4}l^{1/2}\right), & d/\alpha> 2, \\
	\sigma^2n_i+c_3\left(\sum_{l=1}^{k}2^{3l/4}l^{3/2}\right), & d/\alpha = 2,\\
	\sigma^2n_i+c_3\left(\sum_{l=1}^{k}2^{(2\Lambda+3)l/4}l^{1/2-\Lambda}\right), & d/\alpha \in(1,2).
	\end{cases}\\
	&\,\le\, \begin{cases}
	\sigma^2n_i+c_42^{3k/4}k^{3/2}, & d/\alpha\ge 2, \\
	\sigma^2n_i+c_42^{(2\Lambda+3)k/4}k^{1/2}, & d/\alpha \in(1,2).
	\end{cases}\\
	&\,=\,\begin{cases}
	\sigma^2n_i+\mathcal{O}(n_i^{3/4}(\log n_i)^{3/2}), & d/\alpha\ge 2, \\
	\sigma^2n_i+\mathcal{O}(n_i^{(2\Lambda+3)/4}(\log n_i)^{1/2}), & d/\alpha \in(1,2).
	\end{cases}\\
	\end{align*}
Hence, for $d/\alpha>3/2$, 
\begin{equation}\label{eq:q}
\Var\left(\sum_{j=0}^{i-1}\mathcal{V}(n_j,n_{j+1}]\right)\,=\,\sigma^2n_i+\mathsf{o}(n_i).
\end{equation} 
This enables us to apply \Cref{lm:petrov}. We only need to show that 
	$$\sum_{i=2}^\infty\frac{1}{n_i}\mathbb{E}\left[\langle\mathcal{V}(n_i,n_{i+1}]\rangle^2\Ind_{\{|\langle\mathcal{V}(n_i,n_{i+1}]\rangle|\ge\varepsilon\sqrt{n_i/\log\log n_i}\}}\right]\,<\,\infty$$
	as then Kronecker's lemma implies both \cref{eq:11,eq:12}.
	By \Cref{lm:4-th-moment} we obtain
	 $$\mathbb{E}[\langle\mathcal{V}(n_i,n_{i+1}]\rangle^4]\,\le\,c_5(n_{i+1}-n_i)^2\,\le\, c_5n_i$$ 
	 for some $c_5>0$.
Finally, we have
	\begin{align*}
	&\sum_{i=2}^\infty\frac{1}{n_i}\mathbb{E}\left[\langle\mathcal{V}(n_i,n_{i+1}]\rangle^2 \Ind_{\{|\langle\mathcal{V}(n_i,n_{i+1}]\rangle|\ge\varepsilon\sqrt{n_i/\log\log n_i}\}}\right]\\&\,\le\,\sum_{i=2}^\infty\frac{\log\log n_i}{\varepsilon^2n_i^2}\mathbb{E}\left[\langle\mathcal{V}(n_i,n_{i+1}]\rangle^4\right]\\
	&\,\le\, \frac{c_5}{\varepsilon^2}\sum_{i=2}^\infty \frac{\log\log n_i}{n_i}\\
	&\,\le\, \frac{c_5}{\varepsilon^2}\sum_{k=1}^\infty\sum_{l=0}^{k2^{k/2}}\frac{\log\log(2^{k}+l2^{k/2}/k)}{2^{k}+l2^{k/2}/k-1}\\
	&\,\le\,\frac{c_5}{\varepsilon^2}\sum_{k=1}^\infty \frac{k\log\log 2^{k+1}}{2^{k/2}-2^{-k/2}}\,<\,\infty,
	\end{align*}
	which completes the proof.
\end{proof}

The proof of Chung's law of the iterated logarithm is based on the following result. 

\begin{lemma}[{\cite[Theorem A]{Shao}}]\label{lm:} 
Let $\{Y_n\}_{n\ge1}$ be a sequence of independent random variables with mean $0$ and finite variance. Set $$S_n\,=\,\sum_{i=1}^nY_i,\qquad \text{and}\qquad  s^2_n\,=\,\sum_{i=1}^n\mathbb{E}[Y_i^2].$$ 
Suppose that $\lim_{n\nearrow\infty}s_n=\infty$, $\mathbb{E}[Y_n^2]=\mathsf{o}(s_n^2/\log\log s_n^2)$, 
and that  
$\{Y_n^2/\mathbb{E}[Y_n^2]\}_{n\ge1}$ is uniformly integrable. 
 Then
 $$\liminf_{n\nearrow \infty}\frac{\max_{1\le i\le n}|S_i|}{\sqrt{s_n^2/\log\log s_n^2}}\,=\,\frac{\pi}{\sqrt{8}}\qquad \Prob\text{-a.s.}$$
\end{lemma}

\begin{proof}[Proof of \Cref{tm:LIL} - Chung's law of the iterated logarithm] 
We observe that for $n_i\le t<n_{i+1}$ it holds
$$
\sup_{0\le s\le t}|\langle\calV_s\rangle|\,\le\, \max_{0\le j\le i} |\langle\calV_{n_j}\rangle|+\max_{0\le j\le i}\sup_{n_j\le s\le n_{j+1}} |\langle\calV_s\rangle-\langle\calV_{n_j}\rangle|.
$$
We first claim that 
\begin{align}\label{eq_to_show}
\lim_{i\nearrow\infty}\frac{\max_{0\le j\le i}\sup_{n_j\le s\le n_{j+1}} |\langle\calV_s\rangle-\langle\calV_{n_j}\rangle|}{\sqrt{n_i/\log\log n_i}}\,=\,0\qquad\Prob\text{-a.s.}
\end{align}
Indeed, 
according to \cref{eq:13} (if $d/\alpha>1$),  there is $\bar\Omega\subseteq\Omega$ such that $\Prob(\bar\Omega)=1$, and for any $\omega\in\bar\Omega$ and any $\varepsilon>0$ there exists $j_0= j_0(\omega,\varepsilon)\in\N$ for which  
$$
\max_{j\ge j_0} \frac{\sup_{n_j\le s\le n_{j+1}} |\langle\calV_s\rangle(\omega)-\langle\calV_{n_j}\rangle(\omega)|}{\sqrt{n_j/\log\log n_j}}\,\le\,\frac{\varepsilon}{2}.$$ 
We then write
\begin{align*}
	&\frac{\max_{0\le j\le i}\sup_{n_j\le s\le n_{j+1}} |\langle\calV_s\rangle(\omega)-\langle\calV_{n_j}\rangle(\omega)|}{\sqrt{n_i/\log\log n_i}}\\&\,\le\,\frac{\max_{0\le j\le j_0}\sup_{n_j\le s\le n_{j+1}} |\langle\calV_s\rangle(\omega)-\langle\calV_{n_j}\rangle(\omega)|}{\sqrt{n_i/\log\log n_i}}\\&\ \ \ \ \ +\frac{\max_{j_0\le j\le i}\sup_{n_j\le s\le n_{j+1}} |\langle\calV_s\rangle(\omega)-\langle\calV_{n_j}\rangle(\omega)|}{\sqrt{n_i/\log\log n_i}}.
		\end{align*}
	We next choose $i_0=i_0(\omega,\varepsilon)\in\N$ such that
	$$
	\frac{\max_{0\le j\le j_0}\sup_{n_j\le s\le n_{j+1}} |\langle\calV_s\rangle(\omega)-\langle\calV_{n_j}\rangle(\omega)|}{\sqrt{n_i/\log\log n_i}}\,\le\,\frac{\varepsilon}{2},\qquad i\ge i_0,
	$$
and we infer \cref{eq_to_show}.
	We are thus left to show that 
	$$\liminf_{i \nearrow \infty}\frac{\max_{0\le j\le i} |\langle\calV_{n_j}\rangle|}{\sqrt{\sigma^2 n_i/\log\log n_i}}\,=\, \frac{\pi}{\sqrt{8}}\qquad \Prob\text{-a.s.}$$ 
	From \cref{eq:dec} we have 
	\begin{align*}
	\max_{0\le j\le i}\left|\sum_{k=0}^{j-1}\langle\mathcal{V}(n_k,n_{k+1}]\rangle\right|-\max_{0\le j\le i}\left|\sum_{k=0}^{j-1}\langle\mathcal{J}_{n_k}\rangle\right|&\,\le\,\max_{0\le j\le i}|\langle\calV_{n_j}\rangle|\\&\,\le\, \max_{0\le j\le i}\left|\sum_{k=0}^{j-1}\langle\mathcal{V}(n_k,n_{k+1}]\rangle\right|+\max_{0\le j\le i}\left|\sum_{k=0}^{j-1}\langle\mathcal{J}_{n_k}\rangle\right|.
	\end{align*} 
	If we proceed similarly as in the proof of \cref{eq_to_show} and apply \Cref{lm:4.3'} instead of \cref{eq:13}, we arrive at
	$$\lim_{i\nearrow\infty}\frac{\max_{0\le j\le i}|\sum_{k=0}^{j-1}\langle\mathcal{J}_{n_k}\rangle|}{\sqrt{n_i/\log\log n_i}}\,=\,0\qquad\Prob\text{-a.s.}$$
	Thus it suffices to show that 
	$$\liminf_{i \nearrow \infty}\frac{ \max_{0\le j\le i}|\sum_{k=0}^{j-1}\langle\mathcal{V}(n_k,n_{k+1}]\rangle|}{\sqrt{\sigma^2 n_i/\log\log n_i}}\,=\, \frac{\pi}{\sqrt{8}}\qquad \Prob\text{-a.s.}$$
	To this end, we apply \Cref{lm:}. Recall that according to \Cref{lm:4-th-moment}, there is $c>0$ such that $$\mathbb{E}[\langle\mathcal{V}(n_i,n_{i+1}]\rangle^4]\,\le c\,(n_{i+1}-n_i)^2.$$
	This combined  with \Cref{lm:var-asimp} gives that $$\sup_{i\ge1}\frac{\mathbb{E}[\langle\mathcal{V}(n_i,n_{i+1}]\rangle^4]}{\mathbb{E}[\langle\mathcal{V}(n_i,n_{i+1}]\rangle^2]^2}\,<\,\infty$$ 
	which implies that  the sequence $\{\langle\mathcal{V}(n_i,n_{i+1}]\rangle^2/\mathbb{E}[\langle\mathcal{V}(n_i,n_{i+1}]\rangle^2]\}_{i\ge1}$ is uniformly integrable. 
In view of	\Cref{lm:var-asimp} and \cref{eq:2}, 
 $$\mathbb{E}[\langle\mathcal{V}(n_i,n_{i+1}]\rangle^2]\,=\,\mathsf{o}(\sqrt{n_i/\log\log n_i}),$$ 
and the proof is finished.
	\end{proof}

\section{Technical results}\label{Sec:Technical}
\begin{lemma}\label{lm:L1-cvg_with_Levi} 
In the notation of the proof of \Cref{lm:positivity},
for any $k\in\bbN$ it holds that
	\begin{equation*}
		\bbE[\lambda(\calS_1 \setminus \calS[1, N]) \mid \calF_{N - k + 1, N}] \,\xrightarrow[N\nearrow\infty]{{\rm L}^1}\, \bbE[\lambda(\calS_1 \setminus \calS[1, \infty))].
	\end{equation*}
\end{lemma}
\begin{proof}
	Set $m_N = \lambda(\calS_1 \setminus \calS[1, N])$ and $m_{\infty} = \lambda(\calS_1 \setminus \calS[1, \infty))$. We have
	\begin{align*}
		\lim_{N \nearrow \infty}&
		 \bbE[\aps{\bbE[m_N \mid \calF_{N - k + 1, N}] - \bbE[m_{\infty}]}] \\
		& \,\le\, \lim_{N \nearrow \infty} \bbE\big[\aps{\bbE[m_N - m_{\infty} \mid \calF_{N - k + 1, N}]} + \aps{\bbE[m_{\infty} \mid \calF_{N - k + 1, N}] - \bbE[m_{\infty}]}\big] \\
		&\, \le\, \lim_{N \nearrow \infty} \bigl(\bbE[\bbE[\aps{m_N - m_{\infty}} \mid \calF_{N - k + 1, N}]] + \bbE[\aps{\bbE[m_{\infty} \mid \calF_{N -  k + 1, N}] - \bbE[m_{\infty}]}]\bigr) \\
		& \,= \,\lim_{N \nearrow \infty} \bbE[\aps{m_N - m_{\infty}}] + \lim_{N \nearrow \infty} \bbE[\aps{\bbE[m_{\infty} \mid \calF_{N -  k + 1, N}] - \bbE[m_{\infty}]}].
	\end{align*}
	Since $m_N$ clearly converges to $m_{\infty}$ in  ${\rm L}^1$, we are left to prove that
the second limit in the expression above is zero.	
For a fixed $k \in \bbN$ we define
	\begin{equation*}
		\calH_N \,=\, \sigma(X_t : t \ge N - k + 1), \qquad N \ge k.
	\end{equation*}
We observe that $\calF_{N - k + 1, N} \subseteq \calH_N$ and $\calH_N$ is a decreasing family of $\sigma$-algebras. Moreover, according to Kolmogorov's $0-1$ law, for every $H\in \calH_{\infty} = \bigcap_{N \ge 1} \calH_N$, we have $\Prob(H)\in \{0,1\}$. From Levi's theorem (see \cite[Ch.\ II,  Corollary 2.4]{RY_book}) we infer that $\Prob$-a.s.
	\begin{equation}\label{conv_m_infty}
		\lim_{N \nearrow \infty}\bbE[m_{\infty} \mid \calH_N] \,=\, \bbE[m_{\infty} \mid \calH_{\infty}]\, =\, \bbE[m_{\infty}].
	\end{equation}
	Notice that by \cref{eq:LM2.1A}, $\bbE [m_\infty]=\Capa(\sB)$. 
Since the family $\{ |\bbE [m_\infty \mid \calH_N]|\}_{N\ge 1}$ is uniformly integrable, we infer that the convergence in \cref{conv_m_infty} holds also in ${\rm L}^1$, see \cite[Theorem 5.5.1]{Durrett}.
We finally obtain 
	\begin{align*}
		&\lim_{N \nearrow \infty}
		 \bbE\big[\aps{\bbE[m_{\infty} \mid \calF_{N - k + 1, N}] - \bbE[m_{\infty}]}\big] \\
		& \,=\, \lim_{N \nearrow \infty} \bbE\Big[\big\vert \bbE[\bbE[m_{\infty} \mid \calH_N] \mid \calF_{N - k + 1, N}] - \bbE[m_{\infty}]\big\vert \Big] \\
		& \,\le\, \lim_{N \nearrow \infty} \bbE\big[ \bbE[\aps{\bbE[m_{\infty} \mid \calH_N] - \bbE[m_{\infty}]} \mid \calF_{N - k + 1, N}]\big] \\
		& \,=\,  \lim_{N \nearrow \infty} \bbE\Big[\big\vert \bbE[m_{\infty} \mid \calH_N] - \bbE[m_{\infty}]\big\vert \Big] \,=\, 0,
	\end{align*}
	and the proof is finished.
\end{proof}

\begin{lemma}\label{lm:bounded}
	In the notation of the proof of  \Cref{lm:positivity},
	the sequence $\{\frac{1}{n} \sum_{k = 1}^n \bbE[Y_k^2]\}_{n \ge 1}$ is bounded.
\end{lemma}

\begin{proof} 
We set $\Delta = d/\alpha - 3/2$ and recall that it is a positive number.  We present the proof  in the case $\Delta \in (0, 1/2)$ as the proof for  $\Delta \ge 1/2$ is similar.
For $\Delta \in (0, 1/2)$, the function $h(t)$ defined in \cref{eq:def_of_hd} is given by  $h(t)=t^{1/2 - \Delta}$.
By the Cauchy-Schwarz inequality,
\begin{align*}
&\Big\vert \bbE \Big[ \langle\bbE
[\lambda(\calS_n \cap \calS[n, \infty)) \mid \calF_n]\rangle \sum_{k = 1}^n Y_k \Big] \Big\vert \\
& \,\le\, \bigl(\Var(\bbE[\lambda(\calS_n \cap \calS[n, \infty)) \mid \calF_n])\bigr)^{1/2} \Bigg(\sum_{k = 1}^n \bbE[Y_k^2]\Bigg)^{1/2} \\
&\,\le\, 2\sqrt{2}\,c\, n^{1/2 - \Delta} \OBL{\sum_{k = 1}^n \bbE[Y_k^2]}^{1/2}.
\end{align*}
This combined with \cref{al:Var(V_n)} yields
\begin{equation*}
\frac{\Var(\calV_n)}{n} \,\ge\, \frac{1}{n} \Var\bigl(\bbE[\lambda(\calS_n \cap \calS[n, \infty)) \mid \calF_n]\bigr) + \frac{1}{n} \sum_{k = 1}^n \bbE[Y_k^2] - \frac{4\sqrt{2}\,c}{n^{\Delta}} \OBL{\frac{1}{n} \sum_{k = 1}^n \bbE[Y_k^2]}^{1/2}.
\end{equation*}
We suppose, for the sake of contradiction, that there exists a subsequence $\{n_m\}_{m \ge 1}\subseteq\N$ such that
\begin{equation*}
\lim_{m \nearrow \infty} \frac{1}{n_m} \sum_{k = 1}^{n_m} \bbE[Y_k^2] \,=\, \infty.
\end{equation*}
Since 
\begin{equation*}
\lim_{n \nearrow \infty} \frac{\Var(\calV_n)}{n} \,=\, \sigma^2 \qquad \textnormal{and} \qquad \lim_{n \nearrow \infty} \frac{1}{n} \Var\bigl(\bbE[\lambda(\calS_n \cap \calS[n, \infty)) \mid \calF_n]\bigr) \,=\, 0,
\end{equation*}
it follows that
\begin{equation*}
\lim_{m \nearrow \infty} \frac{1}{n_m^{\Delta}} \OBL{\frac{1}{n_m} \sum_{k = 1}^{n_m} \bbE[Y_k^2]}^{1/2} \,=\, \infty.
\end{equation*}
We deduce that $\{\frac{1}{n_m} \sum_{k = 1}^{n_m} \bbE[Y_k^2]\}_{m\ge1}$ diverges faster to infinity than $\{n_m^{2\Delta}\}_{m\ge1}$. 
Since
\begin{equation*}
\lim_{n \nearrow \infty} \frac{\Var(\calV_n)}{n^{1 + 2\Delta}} \,=\, 0,
\end{equation*}
we can again use  \cref{al:Var(V_n)} to obtain
\begin{align*}
\frac{\Var(\calV_{n_m})}{n_m^{1 + 2\Delta}}
& \,\ge\, \frac{1}{n_m^{1 + 2\Delta}} \Var\bigl(\bbE[\lambda(\calS_{n_m} \cap \calS[n_m, \infty)) \mid \calF_{n_m}]\bigr) + \frac{1}{n_m^{1 + 2\Delta}} \sum_{k = 1}^{n_m} \bbE[Y_k^2] \\
& \ \ \ \  - \frac{4\sqrt{2}\,c}{n_m^{3\Delta}} \OBL{\frac{1}{n_m} \sum_{k = 1}^{n_m} \bbE[Y_k^2]}^{1/2}.
\end{align*}
We infer that $\{\frac{1}{n_m} \sum_{k = 1}^{n_m} \bbE[Y_k^2]\}_{m\ge1}$ grows faster to infinity than $\{n_m^{6\Delta}\}_{m\ge1}$. By iterating this procedure, we conclude that
\begin{equation}\label{eq:contradiction}
\lim_{m \nearrow \infty} \frac{1}{n_m^2} \sum_{k = 1}^{n_m} \bbE[Y_k^2] \,= \,\infty.
\end{equation}
On the other hand, from \cref{al:expr_for_Y_k} we have
\begin{align*}
\aps{Y_k}
& \,\le\, \bbE[\lambda(\calS_{k - 1} \cap \calS[k - 1, \infty)) \mid \calF_{k - 1}] + \bbE[\lambda(\calS_{k - 1} \cap \calS[k - 1, \infty)) \mid \calF_k] \\
& \ \ \ \ + \bbE[\lambda(\calS[k - 1, k] \setminus \calS[k, \infty)) \mid \calF_k] + \bbE[\lambda(\calS[k - 1, k] \setminus \calS[k, \infty))] \\
& \,\le\, \bbE[\lambda(\calS_{k - 1} \cap \calS[k - 1, \infty)) \mid \calF_{k - 1}] + \bbE[\lambda(\calS_{k - 1} \cap \calS[k - 1, \infty)) \mid \calF_k] \\
& \ \ \ \  + \lambda(\calS[k - 1, k]) + \Capa(\sB),
\end{align*}
where in the last line we used monotonicity and \cref{eq:LM2.1A}. By Jensen's inequality,
\begin{align*}
\bbE[\aps{Y_k}^2]
& \,\le \,4\bigl(2\,\bbE[\lambda(\calS_{k - 1} \cap \calS[k - 1, \infty))^2]  + \bbE[\lambda(\calS[k - 1, k])^2] + \Capa^2(\sB)\bigr) \\
& \le 64\,c^2h( k - 1)^2 + 4\,\bbE[\calV_1^2] + 4\,(\Capa(\sB))^2\, \le\, c_1 k,
\end{align*}
for a constant $c _1> 0$. This yields 
$\sum_{k = 1}^{n_m} \bbE[Y_k^2]\ \,\le\, c_1n_m^2$,
which gives a contradiction.
\end{proof}

\begin{lemma}
	\label{lm:beta_bound} 
In the notation of the proof of \Cref{lm:positivity},
	for any $\beta\in(0,1]$ there exists a constant $c(d,\alpha,\beta)>0$ such that 
\begin{align*}
	|\phi(y)-\phi(y-x)|\,\le\, c(d,\alpha,\beta)\bigl(\phi(y)+\phi(y-x)\bigr)\left(\frac{1+|x|^\beta}{|y|^\beta}\wedge 1\right),\qquad x,y\in\bbR^d.
\end{align*}
\end{lemma}

\begin{proof} 
Recall that $\phi(y) = \Prob (y\in \calS_\infty)$. This yields
 \begin{equation}\label{eq:triangle}
 |\phi(y)-\phi(y-x)|\,\le\, \phi(y)+\phi(y-x),\qquad x,y\in\bbR^d.
 \end{equation}
 To establish the second non-trivial part of the claimed inequality, that is, for $|y|^\beta>1+|x|^\beta$, we first observe that by rotational invariance of $\X$ it holds $\phi(y) = \Prob_y (\tau_{\sB}<\infty)$. Moreover, by \cref{eq:G*mu=P(T<infty)}, 
\begin{align*}
\phi(y)\,=\, a_{d,\alpha}\int_{\sB}|y-w|^{\alpha-d}\bigl(1-|w|^2\bigr)^{-\alpha/2}\,\D w,\qquad y\in\bbR^d,
\end{align*}
where 
\begin{align*}
a_{d,\alpha}\,=\, \frac{\sin \frac{\pi\alpha}{2} \Gamma ((d-\alpha)/2)\Gamma (d/2)}{2^\alpha \pi^{d+1}\Gamma (\alpha /2)},
\end{align*}
see e.g.\ \cite{Takeuchi}.
We fix $\beta\in(0,1]$ and $x\in\bbR^d$. For any $y\in\sB^c(0,1+|x|)$ we have
\begin{align*}
|y-w|\,\ge\,|y|-|w|\,\ge\,|y|-1\,>\,|x|,\qquad w\in\sB.
\end{align*}
There exists $x_0\in \sB (0,|y-w|)$ lying on the line going through the origin, determined by the vector $y-w$, and such that 
	\begin{align*}\frac{\bigl||y-w|^{\alpha-d}-|y-w-x|^{\alpha-d}\bigr|}{|y-w|^{\alpha-d}+|y-w-x|^{\alpha-d}}\frac{|y-w|^\beta}{1+|x|^\beta}&\,=\,\frac{\bigl||y-w|^{d-\alpha}-|y-w-x|^{d-\alpha}\bigr|}{|y-w|^{d-\alpha}+|y-w-x|^{d-\alpha}}\frac{|y-w|^\beta}{1+|x|^\beta}\\
	&\,\le\,\frac{\bigl||y-w|^{d-\alpha}-|y-w-x_0|^{d-\alpha}\bigr|}{|y-w|^{d-\alpha}+|y-w-x_0|^{d-\alpha}}\frac{|y-w|^\beta}{1+|x_0|^\beta}.\end{align*} 
	Since $x_0$ is necessarily of the form 
	$x_0=\frac{y-w}{|y-w|}\varrho$, for some $\varrho\in[-|y-w|,|y-w|]$, we have
	$$\frac{\bigl||y-w|^{\alpha-d}-|y-w-x|^{\alpha-d}\bigr|}{|y-w|^{\alpha-d}+|y-w-x|^{\alpha-d}}\frac{|y-w|^\beta}{1+|x|^\beta}\,\le\,\frac{\bigl||y-w|^{d-\alpha}-\bigl(|y-w|-\varrho\bigr)^{d-\alpha}\bigr|}{|y-w|^{d-\alpha}+\bigl(|y-w|-\varrho\bigr)^{d-\alpha}}\frac{|y-w|^\beta}{1+|\varrho|^\beta}.
$$
We  investigate the two following cases.

\smallskip

\noindent
\textit{Case 1}.
We first assume that $d-\alpha\le1$. 
If $\varrho\in[0,|y-w|/2]$ then, by the concavity of the function $r\mapsto r^{d-\alpha}$, we obtain
\begin{align*}\frac{\bigl||y-w|^{\alpha-d}-|y-w-x|^{\alpha-d}\bigr|}{|y-w|^{\alpha-d}+|y-w-x|^{\alpha-d}}\frac{|y-w|^\beta}{1+|x|^\beta}&\,\le\,
\frac{(d-\alpha)\varrho\bigl(|y-w|-\varrho\bigr)^{d-\alpha-1}}{|y-w|^{d-\alpha}+\bigl(|y-w|-\varrho\bigr)^{d-\alpha}}\frac{|y-w|^\beta}{1+\varrho^\beta}\\
&\,\le\, (d-\alpha)\frac{\varrho}{|y-w|-\varrho}\frac{|y-w|^\beta}{1+\varrho^\beta}\\
&\,\le\,2(d-\alpha)\frac{\varrho}{|y-w|}\frac{|y-w|^\beta}{1+\varrho^\beta}\\
&\,\le\,2(d-\alpha)\frac{\varrho^\beta}{|y-w|^\beta}\frac{|y-w|^\beta}{1+\varrho^\beta}\\
&\,\le\,2(d-\alpha).
\end{align*}	
If 	$\varrho\in[|y-w|/2,|y-w|]$, then
\begin{align*}\frac{\bigl||y-w|^{\alpha-d}-|y-w-x|^{\alpha-d}\bigr|}{|y-w|^{\alpha-d}+|y-w-x|^{\alpha-d}}\frac{|y-w|^\beta}{1+|x|^\beta}&\,\le\,\frac{|y-w|^\beta}{1+2^{-\beta}|y-w|^{\beta}} \,\le\, 2^{\beta}
\end{align*}
If $\varrho\in[-|y-w|,0]$ then we again use the concavity argument which yields
\begin{align*}
\frac{\bigl||y-w|^{\alpha-d}-|y-w-x|^{\alpha-d}\bigr|}{|y-w|^{\alpha-d}+|y-w-x|^{\alpha-d}}\frac{|y-w|^\beta}{1+|x|^\beta}&\,\le\,
\frac{(d-\alpha)\,|\varrho|\, |y-w|^{d-\alpha-1}}{|y-w|^{d-\alpha}+\bigl(|y-w|-\varrho\bigr)^{d-\alpha}}\frac{|y-w|^\beta}{1+|\varrho|^\beta}\\
&\,\le\, (d-\alpha)\frac{|\varrho|}{|y-w|}\frac{|y-w|^\beta}{1+|\varrho|^\beta}\\
&\,\le\,(d-\alpha)\frac{|\varrho|^\beta}{|y-w|^\beta}\frac{|y-w|^\beta}{1+|\varrho|^\beta}\\
&\,\le\,d-\alpha.
\end{align*}	

\noindent
\textit{Case 2}. Assume that $d-\alpha>1$. 
If $\varrho\in[0,|y-w|]$ then the function $r\mapsto r^{d-\alpha}$ is convex and we obtain
\begin{align*}
\frac{\bigl||y-w|^{\alpha-d}-|y-w-x|^{\alpha-d}\bigr|}{|y-w|^{\alpha-d}+|y-w-x|^{\alpha-d}}\frac{|y-w|^\beta}{1+|x|^\beta}&\,\le\,
\frac{(d-\alpha)\,\varrho\, |y-w|^{d-\alpha-1}}{|y-w|^{d-\alpha}+\bigl(|y-w|-\varrho\bigr)^{d-\alpha}}\frac{|y-w|^\beta}{1+\varrho^\beta}\\
&\,\le\, (d-\alpha)\frac{\varrho}{|y-w|}\frac{|y-w|^\beta}{1+\varrho^\beta}\\
&\,\le\,(d-\alpha)\frac{\varrho^\beta}{|y-w|^\beta}\frac{|y-w|^\beta}{1+\varrho^\beta}\\
&\,\le\,d-\alpha.
\end{align*}
If $\varrho\in[-|y-w|,0]$ then again in view of the convexity 
we have
\begin{align*}
\frac{\bigl||y-w|^{\alpha-d}-|y-w-x|^{\alpha-d}\bigr|}{|y-w|^{\alpha-d}+|y-w-x|^{\alpha-d}}\frac{|y-w|^\beta}{1+|x|^\beta}&\,\le\,
\frac{(d-\alpha)\,|\varrho| \bigl(|y-w|+|\varrho|\bigr)^{d-\alpha-1}}{|y-w|^{d-\alpha}+\bigl(|y-w|+|\varrho|\bigr)^{d-\alpha}}\frac{|y-w|^\beta}{1+|\varrho|^\beta}\\
&\,\le\, (d-\alpha)\frac{|\varrho|}{|y-w|+|\varrho|}\frac{|y-w|^\beta}{1+|\varrho|^\beta}\\
&\,\le\, (d-\alpha)\frac{|\varrho|}{|y-w|}\frac{|y-w|^\beta}{1+|\varrho|^\beta}\\
&\,\le\,(d-\alpha)\frac{|\varrho|^\beta}{|y-w|^\beta}\frac{|y-w|^\beta}{1+|\varrho|^\beta}\\
&\,\le\,d-\alpha.
\end{align*}
Finally, for $y\in\sB^c(0,1+|x|)\cap\sB^c(0,2)$ we obtain
\begin{equation} \label{eq:beta}
\begin{aligned}
\frac{\bigl||y-w|^{\alpha-d}-|y-w-x|^{\alpha-d}\bigr|}{|y-w|^{\alpha-d}+|y-w-x|^{\alpha-d}}&
\frac{|y|^\beta}{1+|x|^\beta}\\
&\,=\,\frac{\bigl||y-w|^{\alpha-d}-|y-w-x|^{\alpha-d}\bigr|}{|y-w|^{\alpha-d}+|y-w-x|^{\alpha-d}}\frac{|y-w|^\beta}{1+|x|^\beta}\frac{|y|^\beta}{|y-w|^\beta}\\
&\,\le\,2^{1+\beta}(d-\alpha)\vee 2^{2\beta}.
\end{aligned}
\end{equation}
On the other hand, if $y\in\sB^c(0,1+|x|)\cap\sB(0,2)$ then $x\in\sB$ and  
\begin{equation}\label{eq:trivial}
\frac{1+|x|^\beta}{|y|^\beta}\,\ge\,2^{-\beta}.
\end{equation}
\Cref{eq:triangle,eq:beta,eq:trivial} imply the result.
	\end{proof}

\begin{lemma}\label{lm:L1}
In the notation of the proof of \Cref{lm:positivity},
it holds that $$\int_{\bbR^d}\int_{\bbR^d}p(1,x)\, \phi(y)\,\bigl|\phi(y)-\phi(y-x)\bigr|\,\D y\,\D x\,<\,\infty.$$
\end{lemma}
\begin{proof}
We split the integral into three parts
\begin{equation}\label{eq:int}
\begin{aligned}
&\int_{\bbR^d}\int_{\bbR^d}p(1,x)\, \phi(y)\,\bigl|\phi(y)-\phi(y-x)\bigr|\,\D y\,\D x\\&\,=\,\int_{\bbR^d}\int_{\sB^c(0,1+|x|)}p(1,x)\, \phi(y)\,\bigl|\phi(y)-\phi(y-x)\bigr|\,\D y\,\D x\\
&\ \ \ \ +\int_{\bbR^d}\int_{\sB(0,1+|x|)\cap\sB}p(1,x)\, \phi(y)\,\bigl|\phi(y)-\phi(y-x)\bigr|\,\D y\,\D x\\
&\ \ \ \ +\int_{\bbR^d}\int_{\sB(0,1+|x|)\cap\sB^c}p(1,x)\, \phi(y)\,\bigl|\phi(y)-\phi(y-x)\bigr|\,\D y\,\D x.
\end{aligned}
\end{equation}  
	According to \Cref{lm:beta_bound}, by setting $\beta=\alpha/2$ we obtain
	$$|\phi(y)-\phi(y-x)|\,\le\, c_1\bigl(\phi(y)+\phi(y-x)\bigr)\frac{1+|x|^{\alpha/2}}{|y|^{\alpha/2}},\qquad y\in\sB^c(0,1+|x|),$$
	where $c_1=c(d,\alpha,\beta)$.
	By  \cite[Lemma 2.5]{Vondra}, there exists a constant $c_2=c_2(d,\alpha)>0$ such that $	\phi(w)\leq c_2|w|^{\alpha -d}$,  for any $w\in\sB^c$.
Thus,
\begin{align*}
&\int_{\bbR^d}\int_{\sB^c(0,1+|x|)}p(1,x)\, \phi(y)\,\bigl|\phi(y)-\phi(y-x)\bigr|\,\D y\,\D x\\
&\,\le\, c_1c_2^2\int_{\bbR^d}\int_{\sB^c(0,1+|x|)}p(1,x)\frac{1}{|y|^{d-\alpha}}\left(\frac{1}{|y|^{d-\alpha}}+\frac{1}{|y-x|^{d-\alpha}}\right)\frac{1+|x|^{\alpha/2}}{|y|^{\alpha/2}}\,\D y\,\D x\\
&\,\le\,2^{d-\alpha/2}(1+2^{d-\alpha})\,c_1c_2^2 \int_{\bbR^d}\int_{\sB^c(0,1+|x|)}p(1,x)\bigl(1+|x|^{\alpha/2}\bigr)
\frac{1}{|y-x|^{2d-3\alpha/2}}\,\D y\,\D x\\
&\,\le\, 2^{d-\alpha/2}(1+2^{d-\alpha})\,c_1c_2^2 \int_{\bbR^d}p(1,x)\bigl(1+|x|^{\alpha/2}\bigr)\,\D x\,\int_{\sB^c}
\frac{1}{|z|^{2d-3\alpha/2}}\,\D z\\
&\,=\, 2^{d-\alpha/2}(1+2^{d-\alpha})\,c_1c_2^2 \, d\, \lambda(\sB)\int_{\bbR^d}p(1,x)\bigl(1+|x|^{\alpha/2}\bigr)\,\D x\,\int_{1}^\infty
\frac{1}{r^{d-3\alpha/2+1}}\,\D r,
\end{align*}
where in the second step we used the fact that $|y-x|\le|y|+|x|\le|y|+|y|-1\le 2|y|$. The last integral is finite as $d/\alpha>3/2$ and $\X$ has finite $\beta$-moment for any $\beta<\alpha$ (see \cite[Example 25.10]{Sato-Book-1999}).

 For the second integral on the right-hand side of \cref{eq:int} we observe that
 $$\int_{\bbR^d}\int_{\sB(0,1+|x|)\cap\sB}p(1,x)\, \phi(y)\,\bigl|\phi(y)-\phi(y-x)\bigr|\,\D y\,\D x\,\le\,2\,\lambda(\sB).$$
 
The third integral on the right-hand side of \cref{eq:int} is most demanding. 
We start by splitting this integral into two parts
 \begin{equation}\label{eq:int2}
 \begin{aligned}
 &\int_{\bbR^d}\int_{\sB(0,1+|x|)\cap\sB^c}p(1,x)\, \phi(y)\,\bigl|\phi(y)-\phi(y-x)\bigr|\,\D y\,\D x\\
 &\,=\,\int_{\sB(0,\Lambda)}\int_{\sB(0,1+|x|)\cap\sB^c}p(1,x)\, \phi(y)\,\bigl|\phi(y)-\phi(y-x)\bigr|\,\D y\,\D x\\
 &\ \ \ \ +\int_{\sB^c(0,\Lambda)}\int_{\sB(0,1+|x|)\cap\sB^c}p(1,x)\, \phi(y)\,\bigl|\phi(y)-\phi(y-x)\bigr|\,\D y\,\D x.\end{aligned}
 \end{equation} 
 For the first integral in this decomposition we have  
 \begin{align*}
 &\int_{\sB(0,\Lambda)}\int_{\sB(0,1+|x|)\cap\sB^c}p(1,x)\, \phi(y)\,\bigl|\phi(y)-\phi(y-x)\bigr|\,\D y\,\D x\\
 &\,\le\, 2\, c_2\int_{\sB(0,\Lambda)}\int_{\sB(0,1+|x|)\cap\sB^c}p(1,x)\, \frac{1}{|y|^{d-\alpha}}\,\D y\,\D x\\
& \,\le\, 2\, c_2\int_{\sB(0,1+\Lambda)\cap\sB^c} \frac{1}{|y|^{d-\alpha}}\,\D y\\
&\,\le\, 2\, c_2\,\lambda(\sB(0,1+\Lambda)).
 \end{align*}
The second integral on the right-hand side of \cref{eq:int2} we decompose further as follows
\begin{equation}\label{eq:int3}
\begin{aligned}
&\int_{\sB^c(0,\Lambda)}\int_{\sB(0,1+|x|)\cap\sB^c}p(1,x)\, \phi(y)\,\bigl|\phi(y)-\phi(y-x)\bigr|\,\D y\,\D x\\
&\,=\, \int_{\sB^c(0,\Lambda)}\int_{\sB(0,1+|x|)\cap\sB^c\cap\{z:|z-x|>|z|\}}p(1,x)\, \phi(y)\,\bigl|\phi(y)-\phi(y-x)\bigr|\,\D y\,\D x\\
&\ \ \ \ + \int_{\sB^c(0,\Lambda)}\int_{\sB(0,1+|x|)\cap\sB^c\cap\{z:1\le|z-x|\le|z|\}}p(1,x)\, \phi(y)\,\bigl|\phi(y)-\phi(y-x)\bigr|\,\D y\,\D x\\
&\ \ \ \ + \int_{\sB^c(0,\Lambda)}\int_{\sB(0,1+|x|)\cap\sB^c\cap\{z:|z-x|<1\}}p(1,x)\, \phi(y)\,\bigl|\phi(y)-\phi(y-x)\bigr|\,\D y\,\D x.
\end{aligned}
\end{equation}
It is well-known that for any $\Lambda>1$ large enough there is $c_3=c_3(d,\alpha,\Lambda)>0$ such that 
$$p(1,x)\,\le\,\frac{c_3}{|x|^{d+\alpha}},\qquad x\in\sB^c(0,\Lambda).$$
We set $A_1 = \sB(0,1+|x|)\cap\sB^c\cap\{z:|z-x|>|z|\}$ and estimate the first integral on the right-hand side of \cref{eq:int3} as follows
 \begin{align*}
&\int_{\sB^c(0,\Lambda)}\int_{A_1}p(1,x)\, \phi(y)\,\bigl|\phi(y)-\phi(y-x)\bigr|\,\D y\,\D x\\
&\,\le\, c_2^2c_3\int_{\sB^c(0,\Lambda)}\int_{A_1}\frac{1}{|x|^{d+\alpha}}\, \frac{1}{|y|^{d-\alpha}}
\left(\frac{1}{|y|^{d-\alpha}}+\frac{1}{|y-x|^{d-\alpha}}\right)\,\D y\,\D x\\
&\,\le\, 2\,c_2^2c_3\int_{\sB^c(0,\Lambda)}\int_{A_1}\frac{1}{|x|^{d+\alpha}}\, \frac{1}{|y|^{2d-2\alpha}}\,\D y\,\D x\\
&\,\le\, 2\,c_2^2c_3\int_{\sB^c(0,\Lambda)}\int_{\sB(0,1+|x|)\cap\sB^c}\frac{1}{|x|^{d+\alpha}}\, \frac{1}{|y|^{2d-2\alpha}}\,\D y\,\D x\\
&\,\le\, 2\,c_2^2c_3\,d\,\lambda(\sB)\int_{\sB^c(0,\Lambda)}\frac{1}{|x|^{d+\alpha}}\int_{1}^{1+|x|} \frac{1}{r^{d-2\alpha+1}}\,\D r\,\D x.
\end{align*}
The last integral is  finite as $d/\alpha>3/2$.

We set $A_2 = \sB(0,1+|x|)\cap\sB^c\cap\{z:1\le|z-x|\le|z|\}$ and for the  second integral on the right-hand side of \cref{eq:int3} we have
\begin{align*}
 &\int_{\sB^c(0,\Lambda)}\int_{A_2}p(1,x)\, \phi(y)\,\bigl|\phi(y)-\phi(y-x)\bigr|\,\D y\,\D x\\
 &\,\le\, c_2^2c_3\int_{\sB^c(0,\Lambda)}\int_{A_2}\frac{1}{|x|^{d+\alpha}}\, \frac{1}{|y|^{d-\alpha}}
\left(\frac{1}{|y|^{d-\alpha}}+\frac{1}{|y-x|^{d-\alpha}}\right)\,\D y\,\D x\\
  &\,\le\, 2^{d+\alpha+1}c_2^2c_3\int_{\sB^c(0,\Lambda)}\int_{A_2}\frac{1}{|y|^{2d}}\, \frac{1}{|y-x|^{d-\alpha}}\,\D y\,\D x\\
  &\,\le\, 2^{d+\alpha+1}c_2^2c_3\int_{\sB^c}\int_{\{z:1\le|z-x|\le|z|\}}\frac{1}{|y|^{2d}}\, \frac{1}{|y-x|^{d-\alpha}}\,\D x\,\D y\\
  &\,\le\, 2^{d+\alpha+1}c_2^2c_3\int_{\sB^c}\frac{1}{|y|^{2d}}\int_{\sB(0,|y|)\cap\sB^c}\, \frac{1}{|z|^{d-\alpha}}\,\D z\,\D y\\
  &\,\le\, 2^{d+\alpha+1}c_2^2c_3\, d\,\alpha^{-1}\lambda(\sB)\int_{\sB^c}\frac{1}{|y|^{2d-\alpha}}\D y\\
  &\,=\, 2^{d+\alpha+1}c_2^2c_3\, d^2\alpha^{-1}\lambda(\sB)^2\int_{1}^\infty\frac{1}{r^{d-\alpha+1}}\D r,
 \end{align*}
 where in the second step we used the fact that $|y|\le 2|x|.$
 
 Finally, we set $A_3= \sB(0,1+|x|)\cap\sB^c\cap\{z:|z-x|<1\}$ and  for the third integral on the right-hand side of \cref{eq:int3} we proceed as follows
  \begin{align*}
& \int_{\sB^c(0,\Lambda)}\int_{A_3}p(1,x)\, \phi(y)\,\bigl|\phi(y)-\phi(y-x)\bigr|\,\D y\,\D x\\
 &\,\le\, 2\, c_2\,c_3\int_{\sB^c(0,\Lambda)}\int_{A_3}\frac{1}{|x|^{d+\alpha}}\, \frac{1}{|y|^{d-\alpha}}\,\D y\,\D x\\
  &\,\le\, 2^{1+d-\alpha} c_2\,c_3\int_{\sB^c(0,\Lambda)}\int_{\{z:|z-x|<1\}} \frac{1}{|x|^{2d}}\,\D y\,\D x\\
 &\,\le\, 2^{1+d-\alpha} c_2\,c_3\,\lambda(\sB)\int_{\Lambda}^\infty \frac{1}{r^{d+1}}\,\D r,
 \end{align*}
 where in the second step we used the fact that $|x|\le |y-x|+|y|\le 1+|y|\le2|y|.$ 
	\end{proof}

\begin{lemma}\label{lm:bound}
In the notation of the proof of \Cref{lm:positivity},
	it holds that 
	\begin{align*}
	\sup_{x\in \sB}\int_{\bbR^d}\phi(y)|\phi(y)-\phi(y-x)|\,\D y\,<\,\infty.
	\end{align*}
		\end{lemma}
	\begin{proof}
We split the integral as follows
 \begin{align*}
 \int_{\bbR^d}\phi(y)|\phi(y)-\phi(y-x)|\,\D y\,=\,&\int_{\sB(0,1+|x|)}\phi(y)|\phi(y)-\phi(y-x)|\,\D y\\&+\int_{\sB^c(0,1+|x|)}\phi(y)|\phi(y)-\phi(y-x)|\,\D y.
 \end{align*}
For any $x\in \sB$ one has $\sB(0,1+|x|)\subseteq\sB(0,2)$, and whence
		$$\sup_{x\in \sB}\int_{\sB(0,1+|x|)}\phi(y)|\phi(y)-\phi(y-x)|\, \D y \,\le\,2\,\lambda(\sB(0,2)).$$ 
By \Cref{lm:beta_bound} with $\beta=\alpha/2$, for a constant $ c_1 = c(d,\alpha,\beta)$,
		$$|\phi(y)-\phi(y-x)|\,\le\, c_1\bigl(\phi(y)+\phi(y-x)\bigr)\frac{1+|x|^{\alpha/2}}{|y|^{\alpha/2}},\qquad y\in\sB^c(0,1+|x|).
		$$
	By \cite[Lemma 2.5]{Vondra}, there exists a constant $c_2=c_2(d,\alpha)>0$ such that $	\phi(w)\leq c_2|w|^{\alpha -d}$,  for any $w\in\sB^c$.
		Thus, for $x \in \sB$ we have
		\begin{align*}
		&\int_{\sB^c(0,1+|x|)}\phi(y)|\phi(y)-\phi(y-x)|\,\D y\\&\,\le\, 2^{d-\alpha/2+1}(1+2^{d-\alpha})\,c_1c_2^2\int_{\sB^c(0,1+|x|)}
		\frac{1}{|y-x|^{2d-3\alpha/2}}\,\D y\\
		&\,\le \, 2^{d-\alpha/2+1}(1+2^{d-\alpha})\,c_1c_2^2\, d\, \lambda(\sB)\int_{1}^\infty
		\frac{1}{r^{d-3\alpha/2+1}}\,\D r,
		\end{align*}	
		where we
	used the fact that $|y-x|\le2|y|.$ The assertion follows as  $d/\alpha>3/2$. 
	\end{proof}

\begin{lemma}\label{lm:4-th-moment}
	There exists a constant $\tilde c > 0$ such that for all $t>0$ large enough,
	\begin{equation*}
		\bbE[\langle \calV_t \rangle^4] \,\le\, \tilde c\,t^2.
	\end{equation*}
\end{lemma}
\begin{proof}
	By setting $n = 2$ in \cref{eq:decom_of_Vt}, we have
	\begin{equation*}
		\calV_t \,=\, \lambda(\calS^{(1)}_{t/2}) + \lambda(\calS^{(2)}_{t/2}) - \lambda\bigl(\calS^{(1)}_{t/2} \cap \calS^{(2)}_{t/2}\bigr),
\end{equation*}
where $\calS^{(1)}_{t/2}$ and $\calS^{(2)}_{t/2}$ are independent, and have the same law as $\calS_{t/2}$. Let $\calV^{(i)}_{t/2} = \lambda(\calS^{(i)}_{t/2})$, for $i = 1, 2$. Taking expectation in the last equation and then subtracting the two relations yields
\begin{equation*}
	\langle \calV_t \rangle \,=\, \langle \calV^{(1)}_{t/2} \rangle + \langle \calV^{(2)}_{t/2} \rangle - \langle \lambda(\calS^{(1)}_{t/2} \cap \calS^{(2)}_{t/2}) \rangle.
\end{equation*}
By the triangle inequality,\footnote{For a random variable $Y$  we write $\norm{Y}_p = (\bbE[\aps{Y}^p])^{1/p}$, for any $p\ge 1$.}
\begin{equation}\label{eq:norm4}
	\norm{\langle \calV_t \rangle}_4 \,\le\, \norm{\langle \calV^{(1)}_{t/2} \rangle + \langle \calV^{(2)}_{t/2} \rangle}_4 + \norm{\langle \lambda(\calS^{(1)}_{t/2} \cap \calS^{(2)}_{t/2}) \rangle}_4,
\end{equation}
Jensen's inequality and \Cref{LM:2.2} imply that there is a constant $c_1>0$ such that
\begin{equation}
	\begin{aligned}\label{al:norm4-2nd-term}
		\norm{\langle \lambda(\calS^{(1)}_{t/2} \cap \calS^{(2)}_{t/2}) \rangle}_4
		& \,\le\, \norm{\lambda(\calS^{(1)}_{t/2} \cap \calS^{(2)}_{t/2})}_4 + \bbE[\lambda(\calS^{(1)}_{t/2} \cap \calS^{(2)}_{t/2})] \,\le\, 2\norm{\lambda(\calS^{(1)}_{t/2} \cap \calS^{(2)}_{t/2})}_4\\
		& \,\le\, 2\norm{\lambda(\calS^{(1)}_{t/2} \cap \calS^{(2)}_{\infty})}_4 \,\le\, c_1 \,h(t/2) \,\le\, c_1 \,\sqrt{t}.
	\end{aligned}
\end{equation}
By the independence of  the variables $\langle \calV^{(1)}_{t/2} \rangle$ and $\langle \calV^{(2)}_{t/2} \rangle$, 
we have
\begin{equation*}
	\bbE\UGL{\OBL{\langle \calV^{(1)}_{t/2} \rangle + \langle \calV^{(2)}_{t/2} \rangle}^4} \,=\, \bbE\UGL{\langle \calV^{(1)}_{t/2} \rangle^4} + \bbE\UGL{\langle \calV^{(2)}_{t/2} \rangle^4} + 6\, \bbE\UGL{\langle \calV^{(1)}_{t/2} \rangle^2} \bbE\UGL{\langle \calV^{(2)}_{t/2} \rangle^2}.
\end{equation*}
By \Cref{lm:positivity}, there exists $N \in \bbN$ large enough such that $\Var(\calV_t) \le c_2 t$ for all $t \ge 2^N$ and some $c_2>0$. Hence, for $t \ge 2^{N + 1}$,
\begin{equation*}
	\bbE\UGL{\OBL{\langle \calV^{(1)}_{t/2} \rangle + \langle \calV^{(2)}_{t/2} \rangle}^4} \,\le\, \bbE\UGL{\langle \calV^{(1)}_{t/2} \rangle^4} + \bbE\UGL{\langle \calV^{(2)}_{t/2} \rangle^4} + 6\OBL{ \frac{c_2 t}{2}}^2.
\end{equation*}
By combining this with the elementary inequality $(a + b)^{1/4} \le a^{1/4} + b^{1/4}$, we arrive at
\begin{equation}\label{eq:norm4-1st-term}
	\norm{\langle \calV^{(1)}_{t/2} \rangle + \langle \calV^{(2)}_{t/2} \rangle}_4 \,\le\, \OBL{\bbE\UGL{\langle \calV^{(1)}_{t/2} \rangle^4} + \bbE\UGL{\langle \calV^{(2)}_{t/2} \rangle^4}}^{1/4} + c_3\sqrt{t}
\end{equation} 
with $c_3=(3c_2^2/2)^{1/4}$.
From \cref{eq:norm4,al:norm4-2nd-term,eq:norm4-1st-term} it follows that there is $c_4>0$ such that
\begin{equation*}
	\norm{\langle \calV_t \rangle}_4 \,\le\, \OBL{\bbE\UGL{\langle \calV^{(1)}_{t/2} \rangle^4} + \bbE\UGL{\langle \calV^{(2)}_{t/2} \rangle^4}}^{1/4} + c_4\sqrt{t}
\end{equation*} 
For $k \ge N$ we set
\begin{equation*}
	\gamma_k \,= \, \sup \{\norm{\langle \calV_t \rangle}_4 : 2^k \le t < 2^{k + 1}\}.
\end{equation*}
Thus, for $k \ge N + 1$ and for every  $2^k \le t < 2^{k + 1}$ we have
\begin{equation*}
	\norm{\langle \calV_t \rangle}_4 \,\le\, (\gamma_{k - 1}^4 + \gamma_{k - 1}^4)^{1/4} + c_5 \,2^{k/2}
\end{equation*} with $c_5=\sqrt{2}\, c_4$.
Taking supremum over $2^k \le t < 2^{k + 1}$  yields
\begin{equation*}
	\gamma_k \,\le\, 2^{1/4} \,\gamma_{k - 1} + c_5\, 2^{k/2}.
\end{equation*}
We set
$\delta_k = \gamma_k/2^{k/2}$ and
we divide the last inequality by $2^{k/2}$. We thus have
\begin{equation*}
	\delta_k \,\le\,  \frac{2^{1/4} \gamma_{k - 1}}{2^{1/2} 2^{(k - 1)/2}} + c_5 \,=\, 2^{-1/4} \delta_{k - 1} + c_5.
\end{equation*}
By iterating this inequality we finally conclude the result.
\end{proof}

\begin{lemma}\label{lm:var-asimp}
The following expansion is valid
\begin{align*}
\Var(\calV_t)\,=\,\sigma^2t+\mathcal{O}(t^{1/2} h(t)),\qquad t\geq 1,
\end{align*}
where the function $h(t)$ is defined in \cref{eq:def_of_hd}.
\end{lemma}
\begin{proof}
	For every $s,t\ge0$ we have 
	$$\calV_{s+t}\,=\,\lambda\bigl(\calS_s\cup\calS[s,s+t]\bigr)\,=\,\lambda(\calS_s^{(1)})+\lambda(\calS_t^{(2)})-\lambda\bigl(\calS_s^{(1)}\cap\calS_t^{(2)}\bigr).$$ 
This implies
	$$\calV^{(1)}_s+\calV^{(2)}_t-\lambda\bigl(\calS_{s+t}^{(1)}\cap\calS_{s+t}^{(2)}\bigr)\,\le\,\calV_{s+t}\,\le\,\calV^{(1)}_s+\calV^{(2)}_t,$$ 
	and whence 
	$$\langle\calV^{(1)}_s\rangle+\langle\calV^{(2)}_t\rangle-\lambda\bigl(\calS_{s+t}^{(1)}\cap\calS_{s+t}^{(2)}\bigr)\,\le\,\langle\calV_{s+t}\rangle\,\le\,\langle\calV^{(1)}_s\rangle+\langle\calV^{(2)}_t\rangle+\bbE[\lambda(\calS_{s+t}^{(1)}\cap\calS_{s+t}^{(2)})].
	$$
Here, $\calS^{(1)}_s$ and $\calS^{(2)}_t$ are independent and have the same law as  $\calS_s$ and $\calS_t$, respectively. 
We set $\mathcal{I}_t\,=\,\lambda\bigl(\calS_t\cap\calS_t'\bigr)$, where $\calS_t'$ is an independent copy of $\calS_t$. From the previous relation we obtain
$$|\langle\calV_{s+t}\rangle-(\langle\calV^{(1)}_s\rangle+\langle\calV^{(2)}_t\rangle)|\,\le\,\mathcal{I}_{s+t}+\bbE[\mathcal{I}_{s+t}]\,\le\, \mathcal{I}_{s+t}+\lVert\mathcal{I}_{s+t}\rVert_2.
$$
Hence
	\begin{equation}\label{eq:var-asimp-key}
\lVert\langle\calV_{s+t}\rangle-(\langle\calV^{(1)}_s\rangle+\langle\calV^{(2)}_t\rangle)\rVert_2\,\le\, 2	\lVert\mathcal{I}_{s+t}\rVert_2,	
	\end{equation}
	and
	\begin{equation}\label{eq:var-asimp-1}
		\lVert \langle\calV_{s+t}\rangle\rVert_2^2\,\le\, \lVert \langle\calV_{s}\rangle\rVert_2^2+\lVert \langle\calV_{t}\rangle\rVert_2^2+4\bigl(\lVert \langle\calV_{s}\rangle\rVert_2^2+\lVert \langle\calV_{t}\rangle\rVert_2^2\bigr)^{1/2}\lVert\mathcal{I}_{s+t}\rVert_2+4\,\lVert\mathcal{I}_{s+t}\rVert_2^2.
	\end{equation}
By \cref{eq:k-th_moment_bound},  there are $c_1 > 0$ and $t_1>1$ such that $\lVert\mathcal{I}_{t}\rVert_2\,\le\,c_1\,h(t)$ for $t\ge t_1.$ For $t\in[1,t_1]$ we clearly have $\mathcal{I}_t \le \mathcal{I}_{t_1}$. Thus, there is a constant $c_2>0$ such that, 
$$\lVert\mathcal{I}_t\rVert_2\,\le\, c_2\, h(t),\qquad t\ge1.$$
Moreover, from \cref{eq:VAR} we have that there exist $c_3 > 0$ and $t_2 > 1$ such that $\Var(\calV_t) \le c_3t$ for $t \ge t_2$, and for $t \in [1, t_2]$ we have $\Var(\calV_t)\le \bbE[\calV_t^2]\le \bbE[\calV_{t_2}^2]\, t$. Hence
	$$\Var(\calV_t)\,\le\, \bigl(c_3+\bbE[\calV_{t_2}^2]\bigr) t,\qquad t\ge 1.$$
We conclude that there is $c_4>0$
such that
	\begin{equation}\label{eq:var-asimp-11}
		\lVert\langle\calV_{t}\rangle\rVert_2\,\le\,c_4\sqrt{t}\qquad\text{and}\qquad \lVert\mathcal{I}_t\rVert_2\,\le\,c_4\, h(t),\qquad t\ge1.	
	\end{equation}
By \cref{eq:var-asimp-1}, we obtain
	\begin{align*}
		\norm{\langle \calV_{s + t} \rangle}^2_2
		& \,\le\, \norm{\langle \calV_s \rangle}^2_2 + \norm{\langle \calV_t \rangle}^2_2 \,+\, 4\,c_4^2 \sqrt{s + t}\, h(s + t) + 4\,c_4^2(h(s + t))^2 \\
		& \,\le\, \norm{\langle \calV_s \rangle}^2_2 + \norm{\langle \calV_t \rangle}^2_2 + c_5 \sqrt{s + t}\, h(s + t),
	\end{align*}
	for some constant $c_5 > 0$. 
	Similarly as above, in view of \cref{eq:var-asimp-key} we have
	\begin{align*}
		\norm{\langle \calV^{(1)}_s \rangle + \langle \calV^{(2)}_t \rangle}_2 
		& \,\le\, \norm{\langle \calV_{s + t} \rangle}_2 + \norm{\langle \calV_{s + t} \rangle - (\langle \calV^{(1)}_s \rangle + \langle \calV^{(2)}_t \rangle)}_2 \\
		& \,\le\, \norm{\langle \calV_{s + t} \rangle}_2 + 2\norm{\calI_{s + t}}_2,
	\end{align*}
	which implies
	\begin{equation*}
		\norm{\langle \calV_s \rangle}^2_2 + \norm{\langle \calV_t \rangle}^2_2 \,\le\, \norm{\langle \calV_{s + t} \rangle}^2_2 + 4\,\norm{\langle \calV_{s + t} \rangle}_2 \norm{\calI_{s + t}}_2 + 4\,\norm{\calI_{s + t}}^2_2.
	\end{equation*}
By \cref{eq:var-asimp-11},
	 \begin{align*}
	 	\norm{\langle \calV_s \rangle}^2_2 + \norm{\langle \calV_t \rangle}^2_2
	 	& \,\le\, \norm{\langle \calV_{s + t} \rangle}^2_2 + 4\,c_4^2 \sqrt{s + t}\,h(s + t) + 4\,c_4^2 (h(s + t))^2 \\
	 	& \,\le\, \norm{\langle \calV_{s + t} \rangle}^2_2 + + c_5 \sqrt{s + t}\,h(s + t).
	 \end{align*}
	 We set
	 \begin{equation*}
	 	x_t \,=\, \Var(\calV_t) \,=\, \norm{\langle \calV_t \rangle}_2^2 \quad \textnormal{and} \quad b_t \,=\, c_5 \sqrt{t}\, h(t), \qquad t > 0,
	 \end{equation*}
and we have shown that
	\begin{equation*}
		x_s + x_t - b_{s + t} \,\le\, x_{s + t} \,\le\,  x_s + x_t + b_{s + t} , \qquad s, t \ge 1.
	\end{equation*}
By \Cref{lm:positivity} we know that
	\begin{equation*}
		\lim_{t \nearrow \infty} \frac{x_t}{t} \,=\, \sigma^2 > 0.
	\end{equation*} 
	Take $s = t = 2^{k-1} r$ for  $k \in \bbN$ and $r \in \bbR$, $r \ge 1$. We easily verify that
	\begin{equation*}
	\left| \frac{x_{2^k r}}{2^k r} - \frac{x_{2^{k-1} r}}{2^{k-1} r} \right| \,\le\, \frac{b_{2^k r}}{2^kr},\qquad k\in\bbN,\ r \ge 1.
	\end{equation*}
	Next, we observe that
	\begin{equation*}
		\sum_{k=1}^{\infty} \left(\frac{x_{2^kr}}{2^kr} - \frac{x_{2^{k-1}r}}{2^{k-1}r}\right) \,=\, \lim_{N \nearrow \infty} \sum_{k = 1}^N \left(\frac{x_{2^kr}}{2^kr} - \frac{x_{2^{k-1}r}}{2^{k-1}r}\right) \,=\, \sigma^2 - \frac{x_r}{r},\qquad r\ge1,
	\end{equation*}
	and whence
	\begin{equation*}
		\left|\frac{x_t}{t} - \sigma^2 \right| \,=\, \left| \sum_{k = 1}^{\infty} \left(\frac{x_{2^kt}}{2^kt} - \frac{x_{2^{k-1}t}}{2^{k-1}t}\right)\right| \,\le\, \sum_{k = 1}^{\infty} \frac{b_{2^kt}}{2^kt},\qquad t\ge1.
	\end{equation*}
This yields
	\begin{equation*}
		\APS{\frac{x_t}{t} - \sigma^2} \,\le\, \sum_{k = 1}^{\infty} \frac{c_5 \sqrt{2^kt}\, h(2^kt)}{2^kt} \,\le\, \frac{c_5}{\sqrt{t}} \sum_{k = 1}^{\infty} \frac{h(2^kt)}{2^{k/2}}, \qquad t \ge 1.
	\end{equation*}
	\textit{Case (i).} 
For $\Delta \in (0,1/2)$ we have
	\begin{equation*}
		\APS{\frac{x_t}{t} - \sigma^2} \,\le\, \frac{c_5}{\sqrt{t}} \sum_{k = 1}^{\infty} \frac{(2^kt)^{1/2 - \Delta}}{2^{k/2}} \,=\, \frac{c_5}{t^{\Delta}} \sum_{k = 1}^{\infty} (2^{-\Delta})^k \,=\, c_6 \,t^{-\Delta}, \qquad t \ge 1,
	\end{equation*}
	where $c_6=c_5\sum_{k=1}^\infty2^{-\Delta\, k}$.
	It follows that
	\begin{equation*}
		\aps{x_t - \sigma^2t} \,\le\, c_6 \,t^{1 - \Delta} \,=\, c_6\, t^{1/2} h(t), \qquad t \ge 1.
	\end{equation*}
	\textit{Case (ii).} 
	If $\Delta \ge 1/2$, then $h(t)$ is slowly  varying. According to \cite[Theorem 1.5.6]{BGT_book} there is a constant $c_7 > 0$ such that $h(2^kt) \le c_7 2^{k/4} h(t)$ for all $k \in \bbN$ and $t \ge 1$. We obtain
	\begin{equation*}
		\APS{\frac{x_t}{t} - \sigma^2}\, \le\, \frac{c_5}{\sqrt{t}} \sum_{k = 1}^{\infty} \frac{c_7 2^{k/4}h(t)}{2^{k/2}} \,=\, \frac{c_5 c_7 h(t)}{\sqrt{t}} \sum_{k = 1}^{\infty} 2^{-k/4} \,=\, c_8\, t^{-1/2} h(t), \qquad t\ge 1,
	\end{equation*} 
	with $c_8=c_5c_7\sum_{k=1}^\infty2^{-k/4}$,
	and the proof is finished.
\end{proof}


\begin{lemma}\label{lm:4.3'} 
Assume that $d/\alpha>9/5$.
	Then, for the process $\{\mathcal{J}_{n_i}\}_{i\geq 0}$ defined in \cref{J_process}, it holds that	$$
	\lim_{i\nearrow\infty}\frac{|\sum_{j=0}^{i-1}\langle\mathcal{J}_{n_j}\rangle |}{\sqrt{n_i/\log\log n_i}}\,=\,0\qquad \Prob\text{-a.s.}
	$$
\end{lemma}
\begin{proof}
	For $i\geq 2$ we clearly  have
$$
	\Var\left(\sum_{j=1}^{i-1}\mathcal{J}_{n_j} \right)\,\le\,\sum_{j=1}^{i-1}\mathbb{E}[\mathcal{J}_{n_j}^2]+2\sum_{j=1}^{i-1} \sum_{k=1}^{j-1}\mathbb{E}[\mathcal{J}_{n_j} \mathcal{J}_{n_k}]
	-
	2\sum_{j=1}^{i-1} \sum_{k=1}^{j-1}\mathbb{E}[\mathcal{J}_{n_j}]\, \mathbb{E}[ \mathcal{J}_{n_k}].$$
	Let $\calS'_t$ be an independent copy of $\calS_t$. Then
	$$\mathcal{J}_{n_j}\,\overset{({\rm d})}{=}\,\lambda(\calS'[0,n_{j+1}-n_{j}]\cap\calS[0,n_j])$$ for $j=1,\dots,i-1$.
	Jensen's inequality and \Cref{Cor:moment_bounds} imply that, for some $c_1>0$,
	$$\mathbb{E}[\mathcal{J}_{n_j}]^2\,\le\,\mathbb{E}[\mathcal{J}_{n_j}^2] \,\le\, c_1 h(n_i)^2.
	$$ 
	Further, for any $j=1,\dots,i-1$ and $k=1,\dots, j-1$, it holds that \begin{align*}
	\mathcal{J}_{n_j}&\,=\,\lambda(\calS(n_{j},n_{j+1}]\cap\calS[0,n_j])\\
	&\,=\,\lambda(\calS(n_{j},n_{j+1}]\cap\calS[n_{k+1},n_j])+\lambda(\calS(n_{j},n_{j+1}]\cap(\calS[0,n_{k+1}]\setminus\calS[n_{k+1},n_j])).
	\end{align*}
	We set
	 \begin{align*}\mathcal{J}^{(1)}_{j,k}&\,=\,\lambda(\calS(n_{j},n_{j+1}]\cap\calS[n_{k+1},n_j]),\\  \mathcal{J}^{(2)}_{j,k}&\,=\, \lambda(\calS(n_{j},n_{j+1}]\cap(\calS[0,n_{k+1}]\setminus\calS[n_{k+1},n_j])).
	 \end{align*}
	Due to independence,  $$\mathbb{E}[\mathcal{J}_{n_k}\mathcal{J}^{(1)}_{j,k}]\,=\,\mathbb{E}[\mathcal{J}_{n_k}]\,\mathbb{E}[\mathcal{J}^{(1)}_{j,k}].$$ Thus,
	$$
	\Var\left(\sum_{j=1}^{i-1}\mathcal{J}_{n_j} \right)\,\le\,\sum_{j=1}^{i-1}\mathbb{E}[\mathcal{J}_{n_j}^2]+2\sum_{j=1}^{i-1} \sum_{k=1}^{j-1}\mathbb{E}[\mathcal{J}_{n_k}\mathcal{J}^{(2)}_{{j,k}} ]
	-
	2\sum_{j=1}^{i-1} \sum_{k=1}^{j-1}\mathbb{E}[ \mathcal{J}_{n_k}]\,\mathbb{E}[\mathcal{J}^{(2)}_{{j,k}}].$$
		Further,
	\begin{align*}
	&\sum_{j=1}^{i-1} \sum_{k=0}^{j-1}\mathcal{J}_{n_k}\mathcal{J}^{(2)}_{j,k}\\&\,=\,\sum_{j=1}^{i-1} \sum_{k=0}^{j-1}\lambda(\calS(n_{k},n_{k+1}]\cap\calS[0,n_k])\, \lambda(\calS(n_{j},n_{j+1}]\cap(\calS[0,n_{k+1}]\setminus\calS[n_{k+1},n_j]))\\
	&\,=\,\sum_{k=0}^{i-2}  \lambda(\calS(n_{k},n_{k+1}]\cap\calS[0,n_k])  \sum_{j=k+1}^{i-1}\lambda(\calS(n_{j},n_{j+1}]\cap(\calS[0,n_{k+1}]\setminus\calS[n_{k+1},n_j]))\\
	&\,\le\, \sum_{k=0}^{i-2}  \lambda(\calS(n_{k},n_{k+1}]\cap\calS[0,n_k])\, \lambda(\calS(n_{k+1},n_{i}]\cap\calS[0,n_{k+1}]).
	\end{align*}
	Similarly as before, by \Cref{Cor:moment_bounds}, we obtain 
	\begin{align*}
	\mathbb{E}[\lambda(\calS(n_{k},n_{k+1}]\cap\calS[0,n_k])^2]&\,\le\,c_1h(n_i)^2,
	\\
	\mathbb{E}[\lambda(\calS(n_{k+1},n_{i}]\cap\calS[0,n_{k+1}])^2]&\,\le\,c_1h(n_i)^2.\end{align*} 
	This implies $$\sum_{j=1}^{i-1} \sum_{k=0}^{j-1}\mathbb{E}[\mathcal{J}_{n_k}\mathcal{J}^{(2)}_{j,k}]\,\le\,c_1 i h(n_i)^2.$$ Analogously we can show that $$\sum_{j=1}^{i-1} \sum_{k=0}^{j-1}\mathbb{E}[\mathcal{J}_{n_k}]\,\mathbb{E}[\mathcal{J}^{(2)}_{j,k}]\,\le\,c_1 i h(n_i)^2.$$
	Thus, $$\Var\left(\sum_{j=0}^{i-1}\mathcal{J}_{n_j} \right)\,\le\,c_2 i h(n_i)^2$$ for some $c_2>0$.
	
	Let $2^k\le n_i< 2^{k+1}$, and $\Lambda=2-d/\alpha$ when $d/\alpha\in(1,2)$. Then
	\begin{equation*}i h(n_i)^2\,\le\, h(2^{k+1})^2\sum_{j=1}^kj2^{j/2}\,=\,\begin{cases}
	\mathcal{O}(n_i^{1/2}\log n_i), & d/\alpha> 2, \\
	\mathcal{O}(n_i^{1/2}(\log n_i)^3), & d/\alpha = 2,\\
	\mathcal{O}(n_i^{2\Lambda+1/2}\log n_i), & d/\alpha \in(1,2),
	\end{cases}
	\end{equation*}
	and, for any $\varepsilon>0$,
	\begin{align*}\Prob\left(\left|\sum_{j=0}^{i-1}\langle\mathcal{J}_{n_j}\rangle\right|\ge\varepsilon\sqrt{n_i/\log\log n_i}\right)&\,=\, \begin{cases}
	\mathcal{O}(n_i^{-1/2}(\log n_i) \log\log n_i), & d/\alpha> 2, \\
	\mathcal{O}(n_i^{-1/2}(\log n_i)^3\log\log n_i), & d/\alpha = 2,\\
	\mathcal{O}(n_i^{2\Lambda-1/2}(\log n_i)\log\log n_i), & d/\alpha \in(1,2)
	\end{cases}\\
	&\,=\, \begin{cases}
	\mathcal{O}(2^{-k/2}k\log k), & d/\alpha> 2, \\
	\mathcal{O}(2^{-k/2}k^3\log k), & d/\alpha = 2,\\
	\mathcal{O}(2^{(2\Lambda-1/2)k}k\log k), & d/\alpha \in(1,2).
	\end{cases}
	\end{align*}
	Let $\epsilon\in(0,1)$ be arbitrary and let us consider a subsequence $\{n_{i_l}\}_{l\ge1}$ which consists of every $\lfloor2^{(1-\epsilon)k/2}\rfloor $-th member of $\{n_i\}_{i\ge0}$ in $[2^k,2^{k+1})$. Clearly, there are at most $k2^{\epsilon k/2}$ members of this subsequence in $[2^k, 2^{k+1}]$.
	We  have 
	\begin{align*}&\sum_{l=1}^\infty\Prob\left(\left|\sum_{j=0}^{i_l-1}\langle\mathcal{J}_{n_j}\rangle\right|\ge\varepsilon\sqrt{n_{i_l}/\log\log n_{i_l}}\right)\\
	&\,\le\,  c_3\begin{cases}
	\sum_{k=1}^\infty2^{(\epsilon-1)k/2}k^2\log k, & d/\alpha> 2, \\
	\sum_{k=1}^\infty 2^{(\epsilon-1)k/2}k^4\log k, & d/\alpha = 2,\\
	\sum_{k=1}^\infty 2^{(4\Lambda+\epsilon-1)k/2}k^2\log k, & d/\alpha \in(1,2)
	\end{cases}
	\end{align*} for some $c_3>0$.
	When $d/\alpha\ge 2$ we take an arbitrary  $\epsilon\in(0,1)$, and when $d/\alpha\in (1,2)$  we take  $\epsilon\in(0,1)$ such that $\epsilon<1-4\Lambda$. Observe that in the former case it is necessary that $\Lambda<1/4$ (that is, $d/\alpha\in(7/4,2)$). In this case,
	$$\sum_{l=1}^\infty\Prob\left(\left|\sum_{j=0}^{i_l-1}\langle\mathcal{J}_{n_j}\rangle\right|\ge\varepsilon\sqrt{n_{i_l}/\log\log n_{i_l}}\right)
	\,<\,\infty.$$ Borel-Cantelli lemma implies that  
	\begin{equation}\label{eq:3}
	\lim_{l\nearrow\infty}\frac{|\sum_{j=0}^{i_l-1}\langle\mathcal{J}_{n_j}\rangle |}{\sqrt{n_{i_l}/\log\log n_{i_l}}}\,=\,0\qquad \Prob\text{-a.s.}
	\end{equation}
We finally prove that \cref{eq:3} holds for the sequence $\{n_i\}_{i\ge0}$. If $2^k\le n_{i_l}\le n_i\le n_{i_{l+1}}\le 2^{k+1}$ then 
	\begin{align*}
	\sum_{j=0}^{i_l-1}\langle\mathcal{J}_{n_j}\rangle-\sum_{j=i_l}^{i_{l+1}-1}\mathbb{E}[\mathcal{J}_{n_j}]\,\le\, \sum_{j=0}^{i-1}\langle\mathcal{J}_{n_j}\rangle\,\le\,\sum_{j=0}^{i_{l+1}-1}\langle\mathcal{J}_{n_j}\rangle+\sum_{j=i_l}^{i_{l+1}-1}\mathbb{E}[\mathcal{J}_{n_j}].
	\end{align*} From \Cref{Cor:moment_bounds} and \cref{eq:2} it follows that $$\mathbb{E}[\mathcal{J}_{n_{j}}]\,=\,\mathcal{O}(h(n_{j+1}-n_j))\,=\, \begin{cases}
	\mathcal{O}(1), & d/\alpha> 2, \\
	\mathcal{O}( \log (n_j^{1/2}/\log n_j)), & d/\alpha = 2,\\
	\mathcal{O}( (n_j^{1/2 }/\log n_j)^\Lambda), & d/\alpha \in(1,2).
	\end{cases}.$$ 
	Hence
	\begin{align*}\sum_{j=i_l}^{i_{l+1}-1}\mathbb{E}[\mathcal{J}_{n_j}]&\,=\,\begin{cases}
	\mathcal{O}(n_i^{(1-\epsilon)/2 }), & d/\alpha> 2, \\
	\mathcal{O}(  n_i^{(1-\epsilon)/2}\log n_i), & d/\alpha = 2,\\
	\mathcal{O}( n_i^{(\Lambda+1-\epsilon)/2 }), & d/\alpha \in(1,2).
	\end{cases}.\end{align*}
	By choosing  an arbitrary $\epsilon\in(0,1)$  in the case when $d/\alpha\ge 2$, and  $\epsilon\in(0,1)\cap(\Lambda,1-4\Lambda)$ in the case when $d/\alpha\in (1,2)$,
	we obtain
	$$\sum_{j=i_l}^{i_{l+1}-1}\mathbb{E}[\mathcal{J}_{n_j}]\,=\,\mathsf{o}(\sqrt{n_i/\log\log n_i}),$$ which concludes the proof. We observe that $\Lambda<1-4\Lambda$ if, and only if, $\Lambda<1/5$, that is, $d/\alpha\in(9/5,2)$.
\end{proof}

\subsection*{Acknowledgement}
This work has been supported by \textit{Deutscher Akademischer Austauschdienst} (DAAD) and \textit{Ministry of Science and Education of the Republic of Croatia} (MSE) via project \textit{Random Time-Change and Jump Processes}. 

Financial support through the \textit{Alexander von Humboldt Foundation} and \textit{Croatian Science Foundation} under projects 8958 and 4197 (for N.\ Sandri\'c), and the \textit{Austrian Science Fund} (FWF) under project P31889-N35 and \textit{Croatian Science Foundation} under project 4197 (for S.\ \v Sebek) is gratefully acknowledged.

\bibliographystyle{abbrv}
\bibliography{BIB}

\end{document}